\documentclass{article}
\usepackage{fullpage}
\usepackage[utf8]{inputenc}
\usepackage[french,english]{babel}
\usepackage{amsmath,amsfonts,pifont,amssymb,latexsym,amsthm}
\usepackage{bbm} 
\usepackage{hyperref}
\usepackage{graphicx}
\usepackage{tikz}
\usetikzlibrary{positioning,shadows.blur}
\usepackage{comment}

\newcommand\Z{\mathbbm Z}
\newcommand\N{\mathbbm N}

\newcommand\R{\mathbbm R}

\newcommand\E{\mathcal E}
\newcommand\F{\mathcal F}
\newcommand\FF{\mathfrak F}
\newcommand\BF{\mathcal B}
\newcommand\B{B}
\newcommand\supp[1]{\mathcal S_{#1}}
\newcommand\AN[2][]{\mathcal D_{#1}({#2})}
\newcommand\FN[2][]{\mathfrak d_{#1}({#2})}

\newcommand{\restr}[1]{_{\left|#1\right.}}
\newcommand{\compl}[1]{{#1}^C}
\usepackage{stmaryrd} 
\newcommand{\cc}[2]{\left\llbracket #1,#2\right\rrbracket}
\newcommand{\co}[2]{\left\llbracket #1,#2\right\llbracket}
\newcommand{\ic}[1]{\left\rrbracket-\infty,#1\right\rrbracket}
\newcommand{\ci}[1]{\left\llbracket #1,+\infty\right\llbracket}
\newcommand{\oin}[1]{\mathcal I(#1)}

\newcommand{\soit}[1]{\left|\everymath{\displaystyle\everymath{}}\begin{array}{ll}#1\end{array}\right.}
\newcommand{\both}[1]{\left\{\begin{array}{l}#1\end{array}\right.}

\newcommand\Min{\texttt{Min}}

\newcommand{\dinf}[1]{\vphantom{#1}^\infty{#1}^\infty}

\newcommand{\cl}[1]{\overline{#1}}

\newcommand{\sett}[2]{\left\{\left.#1\vphantom{#2}\right|#2\right\}}
\newcommand{\set}[3]{\sett{#1\in#2}{#3}}
\newcommand{\card}[1]{\left|#1\right|}
\newcommand{\len}[1]{\left|#1\right|}
\newcommand{\abs}[1]{\left|#1\right|}
\newcommand{\ipart}[1]{\left\lfloor#1\right\rfloor}
\newcommand{\orb}{\mathcal O}

\newcommand\asym[2][]{\mathcal A_{#1}(#2)}
\newcommand\lang{\mathcal L}
\newcommand\pprec{\prec\!\!\prec}

\newtheorem{thm}{Theorem}[section]
\newtheorem{prop}[thm]{Proposition}
\newtheorem{defi}[thm]{Definition}
\newtheorem{exmp}[thm]{Example}
\newtheorem{rmq}[thm]{Remark}
\newtheorem{crl}[thm]{Corollary}
\newtheorem{lem}[thm]{Lemma}
\newcommand{\et}{\textrm{ and }}
\newcommand{\oub}{\textrm{ or }}
\newcommand{\ou}{\textrm{ where }}
\newcommand{\si}{\textrm{ if }}
\newcommand{\sinon}{\textrm{ otherwise}}
\newcommand\dfn[1]{\textbf{#1}}
\newcommand{\resp}[1]{\ (resp. #1)}
\newcommand\rpm[1]{\mathbf{r_{#1}}}
\newcommand\pierre{\color{blue}}
\newcommand\saliha{\color{red}}

\usepackage{authblk}
\title{The generic limit set of cellular automata}
\author[$\star$]{Saliha Djenaoui}
\author[$\dagger$]{Pierre Guillon}
\affil[$\dagger\star$]{Universit\'e d'Aix-Marseille CNRS, Centrale Marseille, I2M, UMR 7373,\\ 13453 Marseille, France}
\affil[$\star$]{Département de Mathématiques, Universit\'e Badji Mokhtar-Annaba, B.P. 12, Sidi Amar, 23220 Annaba, Algérie\\
	\texttt{sdjenaoui1@yahoo.fr}}
\affil[$\dagger$]{Interdisciplinary Scientific Center J.-V. Poncelet (ISCP), Independent University of Moscow, CNRS, UMI 2615, Russia\\
	\texttt{pierre.guillon@math.cnrs.fr}}

\begin{document}
	
	\maketitle
	\begin{abstract}
		In topological dynamics, the generic limit set is the smallest closed subset which has a comeager realm of attraction. We study some of its topological properties, and the links with equicontinuity and sensitivity. 
		We emphasize the case of cellular automata, for which the generic limit set is included in all subshift attractors, and discuss directional dynamics, 
		as well as the link with measure-theoretical similar notions.
	\end{abstract}
	
	
	\section{Introduction}
	In a topological dynamical system (DS), the limit set is the set of points that appear arbitrarily late during the evolution (see \cite{Culik}). 
	But it may include points which look transient, because they do not appear arbitrarily late around any orbit. 
	
	J.~Milnor, interested in the dynamics on the space of measures, introduced in \cite{milnor} the notion of likely limit set, that provides a useful tool for studying asymptotic behavior for almost all orbits. 
	He also implicitly defined a topological version of the same intuitive idea, that he calls the \emph{generic limit set}.
	The goal of our article is to formalize this concept. In other words, we focus on the asymptotic behavior for almost all orbits in the sense of Baire category theory. We study some topological properties of the generic limit set, which is the smallest closed set that has a comeager realm. 
	
	We show that the generic limit set is actually equal to the limit set if the DS is semi-nonwandering (Proposition~\ref{p:snwand}), a broad property that is implied by nonwanderingness, or equicontinuous (Proposition~\ref{p:equi}). We also prove that the generic limit set is the closure of the asymptotic set of the equicontinuity points if the DS is almost equicontinuous (Proposition \ref{p:eqan}). 
	
	For cellular automata (CA), we know that all subshift attractors have a dense open realm of attraction (see \cite{hurley, formenti, Petr, Mike}). We prove that the generic limit set is a subshift (Proposition~\ref{p:subsh}), which is included in all subshift attractors (Corollary~\ref{c:subattr}). 
	
	We emphasize directional dynamics of cellular automata, which is devoted to their qualitative behaviour (equicontinuity, sensitivity, expansiveness) when composed with shifts (see \cite{sablik,Martin}). 
	First, in oblique directions (the shifts are bigger than the radius), a weak semi-mixingness property makes the generic limit set equal to the limit set. 
	If the generic limit set is finite, it is shown that it consists of only one periodic orbit of a monochrome configuration (Proposition~\ref{p:periodicorbit}) and the DS is almost equicontinuous (Proposition~\ref{p:sensinfini}). We show that it is the case if the cellular automaton is almost equicontinuous in two directions of opposite sign; moreover the realm of this periodic orbit then contains a dense open set (Proposition~\ref{p:alldir}). We mention a nontrivial example where the period is nontrivial (Example \ref{e:martin}).
	We give a classification of generic limit sets in the context of directional dynamics (Theorem \ref{t:class}).
	
	We formulate most topological-dynamical results in a very general framework of sequences of continuous functions, which correspond to nonuniform dynamical systems. The purpose is double: to be able to apply our results to directional dynamics of cellular automata in a smoother way than previous works (which often had to introduce several ad-hoc definitions), and hopefully to propose a large setting in which different kinds of attractor properties can be studied, that could be useful in other subcases.
	
	The paper is structured as follows : In Section~\ref{s:preliminaries}, we provide the basic background on the subject of topological dynamical systems and cellular automata. In Section~\ref{s:attractor}, we show some preliminary results about attractors, limit sets, realms. In Section~\ref{s:genericsystem}, we define the generic limit set and prove the main results about it. In Section~\ref{s:directional}, 
	we show some consequences on directional dynamics of cellular automata, and provide a list of examples. 
	In Section~\ref{s:measure}, we compare the generic limit set with the likely limit set. 
	
	\section{Preliminaries}\label{s:preliminaries}
	\subsection{Topology}
	In this article, $(X,d)$ is a compact metric space. We put $\B_{\delta}(x)=\set yX{d(x,y)<\delta}$ and call it the \dfn{open ball} with center $x\in X$ and radius $\delta>0$. 
	{For $U,V\subseteq X, we note d(U,V)=\inf\sett{d(x,y)}{x \in U, y \in V}$}. We may write $d(U,y)=d(U,\{y\})$ for $y\in Y$.
	We also write $\B_\delta(U)=\set yX{d(U,y)<\delta}$ and $\cl\B_\delta(U)$ its closure.
	
	A subset $U\subseteq X$ is called \dfn{comeager} in $X$ if it includes a countable intersection of dense open sets. A subset $U$ is \dfn{meager} if $X\backslash U$ is comeager in $X$. By the Baire category theorem, the intersection of countably many dense open sets in $X$ is dense in $X$. Hence, a comeager set is dense in $X$.
	
	We also say that a set $U\subseteq X$ is \dfn{comeager in} some set $V\subseteq X$ if $U\cap V$ is comeager in the induced topological space $V$.
	
	A subset $A$ of $X$ is said to have the \dfn{Baire} property if there is an open set $U$ such that the symmetric difference $A\Delta U=(A\backslash U)\cup (U\backslash A)$ is meager in $X$. {The family of sets with the Baire property forms a $\sigma$-algebra.}
	Every Borel subset 
	has the Baire property (for more details, see for instance \cite{Srivastava}).
	
	We recall the following folklore remark, further used several times.
	\begin{rmq}\label{r:Borel}
		A set $W\subseteq X$ with the Baire property is not comeager if and only if there exists a nonempty open set $U$ in which $W\cap U$ is meager.
	\end{rmq}
	\begin{proof}
		$\compl W\Delta U$ is meager for some open set $U$, and $W\cap U\subseteq\compl W\Delta U$.
		Just remark that $W$ is comeager if and only if $U=\emptyset$. 
	\end{proof}
	\subsection{Topological dynamics}
	Now we introduce some key concepts of the topological dynamics.
	A (time-nonuniform) \dfn{dynamical system} (DS) is any sequence $\FF=(F_t)_{t\in\N}$ of continuous self-maps of {some compact metric space }$X$.
	This general formalism will be useful when studying directional dynamics of cellular automata, but the reader should keep in mind the following specific, more classical, case.
	$\FF$ is \dfn{uniform} if $F_t=G^t$, for all $t\in\N$ and some continuous self-map $G$ of $X$.
	We may then write the uniform DS simply as $G$.
	
	We are interested in 
	the \dfn{orbits} $\orb_\FF(x)=\sett{F_t(x)}{t\in\N}$ of points $x\in X$.
	We say that a set $U$ is $\FF$-invariant\resp{strongly} if $F_{t+1}(U)\subseteq F_t(U)$\resp{$F_t^{-1}(U)=U=F_t(U)$}, for every $t\in\N$.
	
	\paragraph{Equicontinuity.}
	For $\varepsilon>0$, a point $x\in X$ is \dfn{$\varepsilon$-stable} if there exists $\delta>0$ such that $\forall y\in\B_\delta(x),\forall t\in\N,d(F_t(x),F_t(y))<\varepsilon$.
	The set $\E_\FF \subseteq X$ of \dfn{equicontinuous} points for $\FF$ is the set of points which are $\varepsilon$-stable for every $\varepsilon>0$. 
	If $\E_\FF$ is comeager, then we say that $\FF$ is \dfn{almost equicontinuous}.
	If $\E_\FF = X$, then we say that $\FF$ is \dfn{equicontinuous}.
	Equivalently by compactness, for every $\varepsilon>0$, there is a uniform $\delta>0$ such that $\forall x\in X,\forall y\in\B_\delta(x),\forall t\in\N,d(F_t(x),F_t(y))<\varepsilon$.
	
	$\FF$ is \dfn{sensitive} if there exists $\varepsilon > 0$ such that $\forall x \in X, \forall \delta > 0, \exists y \in \B_{\delta}(x), \exists t\in\N, d(F_t(x),F_t(y)) \geq \varepsilon$.
	This implies that $\E_\FF=\emptyset$.
	
	\begin{rmq}\label{r:eqstab}
		$\E_\FF$ is comeager in some nonempty open set $U\subseteq X$ if and only if, for every $\varepsilon>0$, the set of $\varepsilon$-stable points is.
	\end{rmq}
	\begin{proof}
		$\E_\FF$ can be written as the decreasing countable intersection of the sets of $1/n$-stable points.
	\end{proof}
	
	A class which is between nonuniform and uniform DS is the following: $\FF$ is \dfn{semi-uniform} if $F_t=G_{t}\cdots G_1$, for all $t\in\N$ and some equicontinuous sequence $(G_t)_{t\ge1}$ of self-maps of $X$.
	This is trivially satisfied when $\sett{G_t}{t\ge1}$ is finite; in particular, in the case of uniform DS: $G_t=G_1$ for every $t\ge1$.
	We will sometimes denote $G_{T+\cc1t}$ the composition $G_{T+t}\cdots G_{T+2}G_{T+1}$, so that $F_{t+T}=G_{T+\cc1t}F_T$.
	
	\paragraph{Transitivity.}
	Classical notions from topological dynamics can be adapted in our framework of nonuniform DS.
	Here is an example, that will actually be used especially for a uniform DS, the shift map from the next section, but we give the nonuniform version for completeness, and for the user to get used to the little differences that hold in this setting, in the perspective of better understanding related notions, like the ones introduced in Subsection~\ref{s:oblique}.
	A DS $\FF$ is said to be \dfn{transitive}\resp{\dfn{weakly mixing}} if for any nonempty open subsets $U,V$\resp{and $U',V'$} of $X$ and for every $T\in\N$, there exists $t\ge T$ such that $F_t^{-1}(U)\cap V\ne\emptyset$\resp{and $F_t^{-1}(U')\cap V'\ne\emptyset$}.
	One can prove that, if the space is perfect, then it is enough to suppose this for $T=0$.
	For any uniform DS, it is classical that it is enough to suppose this for $T=1$ (and for $T=0$ if the uniform DS is surjective).
	
	The following lemma will be used in the context of shift maps: it can be interpreted as the fact that a transitive DS \emph{mixes} the space, in the sense that it transforms a \emph{local} topological property into a global one.
	\begin{lem}\label{l:comeager}
		Let $\FF=(F_t)_{t\in\N}$ be a transitive DS, where all $F_t$ are homeomorphisms, and let $W\subseteq X$ be a strongly $\FF$-invariant subset.
		\begin{enumerate}
			\item\label{i:densinu} $W$ is either dense or nowhere dense.
			\item\label{i:densopen} $W$ either has empty interior or includes a dense open set.
			\item\label{i:meagerco} If $W$ has the Baire property, then it is either meager or comeager.
			\item\label{i:coseq} If $\bigcup_{i\in\N}W_i$ has nonempty interior\resp{is not meager}, where each $W_i$ is strongly $\FF$-invariant \resp{and has the Baire property}, then there exists $i\in\N$ such that $W_i$ includes a dense open set\resp{is comeager}.
		\end{enumerate}
	\end{lem}
	\begin{proof}
		~\begin{enumerate}
			\item Suppose $W$ is dense in some nonempty open set $U$. 
			Since $\FF$ is transitive, for every nonempty open set $V$, there exists $t$ such that $F_t^{-1}(V)\cap U$ is a nonempty open set. Moreover, $W$ is dense in $U$, so that $F_t^{-1}(V)\cap U\cap W\ne\emptyset$. Since $F_t(W)\subseteq W$, $V\cap F_t(U)\cap W\supseteq F_t(F_t^{-1}(V)\cap U\cap W)\ne\emptyset$.
			So, $W$ is dense in $X$.
			\item Suppose that $W$ includes a nonempty open set and is strongly $\FF$-invariant; then $W$ includes the (nonempty open) orbit of this open set. From Point~\ref{i:densinu}, this orbit is a dense open set.
			\item Suppose that $W$ is not meager and consider any nonempty open subset $V\subseteq X$.
			By Remark~\ref{r:Borel}, there exists an open set $U$ such that $W$ is comeager in $U$.
			By transitivity, there exists $t\in\N$ such that $F_t^{-1}(U)\cap V$ is a nonempty open subset.
			By assumption, $W$ is not meager in its nonempty open (because $F_t$ is a homeomorphism) image $U\cap F_t(V)$, and since $F_t^{-1}$ is a homeomorphism, $F_t^{-1}(W)$ is not meager in $F_t^{-1}(U)\cap V$.
			By reverse invariance, we get that $W$ is not meager in $F_t^{-1}(U)\cap V$.
			By Remark~\ref{r:Borel}, since $W$ is not meager in any nonempty open set, it is comeager.
			\item One of the $W_i$ has to have nonempty interior\resp{to not be meager}. We conclude by the previous point.
			\popQED	\end{enumerate}
	\end{proof}
	
	\subsection{Symbolic dynamics}
	\paragraph{Configurations.}
	Let $A$ be a finite set called the \dfn{alphabet}. A word over $A$ is any finite sequence of elements of $A$. Denote $A^*=\bigcup_{n\in\N}A^n$ the set of all finite words $u=u_0\dots u_{n-1} ; \len u=n$ is the length of $u$.
	We say that $v$ is a \dfn{subword} of $u$ and write $v\sqsubset u$, if there are $k,l<\len u$ such that $v=u_{\co kl}$.
	$A^{\Z}$ is the \dfn{space of configurations}, equipped with the following metric:
	\[d(x,y):=2^{-n},\text{ where }n=\min\sett{i\in\N}{x_i\ne y_i\text{ or }x_{-i}\ne y_{-i}}.\]
	$A^{\Z}$ is a Cantor space.
	The \dfn{cylinder} of $u\in A^*$ in position $i$ is $[u]_i=\set x{A^{\Z}}{x_{\co i{i+|u|}}=u}$. Cylinders are clopen (closed and open). The (full) \dfn{shift} is the dynamical system $\sigma$ over space $A^\Z$ defined by $\sigma(x)_i = x_{i+1}$ for $i \in \Z$ and $x\in A^\Z$. It is $2$-Lipschitz. 
	
	The (spatially) \dfn{periodic} configuration $\dinf u$ is defined by $(\dinf u)_{k|u|+i}=u_i$ for $k\in\Z, 0\leq i<|u|$ and $u\in A^+$.
	A \dfn{monochrome} configuration is one with only one symbol: $\dinf0$, for some $0\in A$.
	
	\paragraph{Subshifts.}
	A \dfn{subshift} is any subsystem $\sigma\restr\Sigma$ of the full shift; we usually simply denote it by $\Sigma$, which is then simply a closed strongly $\sigma$-invariant subset of $A^{\Z}$.
	Equivalently, there exists a \dfn{forbidden language} $\F\subseteq A^*$ such that $\Sigma=\Sigma_{\F}=\set x{A^{\Z}}{\forall u\sqsubset x,u\notin \F}$.
	If $\F$ can be taken finite, then one says that $\Sigma_{\F}$ is a \dfn{subshift of finite type} (SFT); in that case $\F$ can be taken included in $A^k$ for some $k\in\N$, which is an \dfn{order} for the SFT. We write that $\Sigma_{\F}$ is a $k$-SFT.
	Let $\Sigma\subseteq A^{\Z}$ be a subshift. Then $\lang(\Sigma)=\set u{A^*}{\exists x\in\Sigma, u\sqsubset x}$ is the \dfn{language} of $\Sigma$.
	
	\begin{rmq}\label{r:comeager}
		It is clear that $(A^\Z,\sigma)$ is transitive.
		Point~\ref{i:meagerco} of Lemma~\ref{l:comeager} hence applies, and gives that nonfull subshifts are all nowhere dense (hence meager).
	\end{rmq}
	
	We shall use the following lemma to show results concerning the dynamics of cellular automata.
	\begin{lem}\label{l:lipschmodSFT}
		Let $\varepsilon>0$ and $V\subseteq A^\Z$. ~\begin{enumerate}
			\item 
			$\sigma^j(\B_{\varepsilon}(V))\subseteq \B_{2^{j}\varepsilon}(\sigma^j(V))$ for all $j\in\Z$.
			\item \label{i:lipschmod} If $V$ is strongly $\sigma$-invariant, $x\in A^\Z$, and $p>0$ are such that for all $n\in\Z$, $\sigma^{pn}(x)\in \B_\varepsilon(V)$, then $\forall i\in \Z, \sigma^i(x)\in \B_{2^{p}\varepsilon}(V)$.
			\item \label{i:SFT}
			If $V$ is a $2k+1$-SFT, then $\bigcap_{i\in\Z}\sigma^i(\B_{2^{-k}}(V))=V$.
	\end{enumerate}\end{lem}
	\begin{proof}~\begin{enumerate}
			\item 
			Let $x\in\sigma^j(\B_{\varepsilon}(V))$. Then $\sigma^{-j}(x)\in\B_{\varepsilon}(V)$, which means that $d(\sigma^{-j}(x),V)<\varepsilon$.
			Moreover, $d(x,\sigma^j(V))=d(\sigma^j\sigma^{-j}(x), \sigma^j(V))<2^jd(\sigma^{-j}(x), V)<2^j\varepsilon$. So, $x\in\B_{2^{j}\varepsilon}(\sigma^j(V))$. 
			\item
			From the previous point, $\forall i\in \Z, \sigma^i(x)=\sigma^{i\bmod p}\sigma^{\lfloor i/p\rfloor p}(x)\in \sigma^{i\bmod p}(\B_\varepsilon(V))\subseteq\B_{2^{i\bmod p}\varepsilon}(\sigma^{i\bmod p}(V))$. Since $V$ is strongly $\sigma$-invariant, we also have that $\B_{2^{i\bmod p}\varepsilon}(\sigma^{i\bmod p}(V))\subseteq\B_{2^{p}\varepsilon}(V)$. Hence, $\forall i\in \Z, \sigma^i(x)\in \B_{2^{p}\varepsilon}(V)$.
			\item Let $i\in\Z$.
			\begin{align*}
			\sigma^{-i}(x)\in \B_{2^{-k}}(V)
			&\iff d(\sigma^{-i}(x),V)<2^{-k}\\
			&\iff\exists y\in V,\sigma^{-i}(x)_{\cc{-k}k}=y_{\cc{-k}k}\\
			&\iff\exists y\in V,x_{\cc{-i-k}{-i+k}}=y_{\cc{-k}k}\\
			&\iff x_{\cc{-i-k}{-i+k}}\in\lang (V)~.
			\end{align*}
			If this is true for every $i\in\Z$ and $V$ is $(2k+1)$-SFT, then $x\in V$.
			\popQED\end{enumerate}\end{proof}
	
	\subsection{Cellular automata}
	A map $F : A^{\Z} \to A^{\Z}$ is a \dfn{cellular automaton} (CA) if there exist integers $r_-\leq r_+$ (\dfn{memory} and \dfn{anticipation}) and a local rule $f : A^{r_+-r_-+1} \to A$ such that for any $x \in A^{\Z}$ and any $i \in \Z, F(x)_i = f(x_{i+r_-},\dots,x_{i+r_+})$. $d=r_+-r_-+1\in\N$ is sometimes called the \dfn{diameter} of $F$.
	Sometimes we assume that $-r_-=r_+$, which is then called the \dfn{radius} of $F$ (it is always possible to obtain this, by taking $r=max\{|r_-|,|r_+|\}\in\N$). 
	By Curtis, Hedlund and Lyndon \cite{Hedlund}, a map $F : A^{\Z} \to A^{\Z}$ is a CA if and only if it is continuous and commutes with the shift.
	In particular, CA induce uniform DS over $A^\Z$.
	
	\paragraph{Directional dynamics.}
	We call \dfn{curve} a map $h:\N\to\Z$ with bounded variation, that is: $M_h=\sup_{t\in\N}\abs{h(t+1)-h(t)}$ is finite.
	The map is meant to give a position in space for each time step.
	Following \cite{Martin}, the \dfn{CA $F$ in direction $h$} will refer to the sequence $(F^t\sigma^{h(t)})_{t\in\N}$. We will use all notations for DS with $F,h$ instead of $\FF$, when dealing with it (for instance $\E_{F,h}$ is its set of equicontinuous points).
	
	In a first reading, the reader can understand the next definitions and results by considering the classical case: $h$ constantly $0$.
	In general, the directional dynamics of a CA can be read on its space-time diagram, by following $h$ as a curve when going in the time direction.
	An example of curve is given by the (possibly irrational) lines: $\alpha\in\R$ will stand for the direction $t\mapsto\ipart{t\alpha}$.
	The dynamics \emph{along $\alpha$} then corresponds to that studied in \cite{sablik, zana}.
	
	
	\paragraph{Equicontinuity.}
	A word $u\in A^*$ is (strongly) \dfn{blocking} for a CA $F$ along curve $h$ if there exists an offset $s\in\Z$ such that for every $x,y\in[u]_s$, $\forall t\in\N,F^t\sigma^{h(t)}(x)_{\co0{M}}=F^t\sigma^{h(t)}(y)_{\co0{M}}$, where $M=\max(-r_-+\max_t(h(t)-h(t+1)),r_++\max_t(h(t+1)-h(t)))$, and $r_-$ and $r_+$ are the (minimal) memory and anticipation for $F$.
	The terminology comes from the fact that in that case, $u$ is both left- and right-blocking (with the same offset), which is taken as a definition in \cite{Martin}:
	A word $u\in A^*$ is \dfn{right-blocking} for a CA $F$ in direction $h$ if there exists an offset $s\in\Z$ such that: 
	\[\forall x,y\in[u]_s,x_{\ic s}=y_{\ic s}\implies\forall t\in\N,F^t(x)_{\ic{h(t)}}=F^t(y)_{\ic{h(t)}}~.\]
	We define \dfn{left-blocking} words similarly. 
	
	The following proposition explains how equicontinuity in cellular automata can be rephrased in terms of blocking words.
	The vertical case dates back from \cite{Petr,kurka}, the linear directions from \cite{sablik}, the directions with bounded variations can be found in the proofs of \cite[Prop~2.1]{Martin}; a version with unbounded variations of the first point can even be found in \cite[Prop~3.1.3, Cor~3.1.4]{Martinthese}.
	\begin{prop}\label{p:wall}
		Let $F$ be a CA and $h$ a curve.
		\begin{enumerate}
			\item If there is a left- and right-blocking word $u$ for $F$ in direction $h$, then $\E_{F,h}$ includes the comeager set of configurations where $u$ appears infinitely many times on both sides.
			\item Otherwise, 
			$F$ is sensitive in direction $h$.
		\end{enumerate}
	\end{prop}
	In particular, $\E_{F,h}$ is either empty or comeager.
	The question is open whether this remains true in the unbounded-variation case (see \cite[Rem~3.1.1]{Martinthese}).
	
	\begin{defi}[{\cite[Def~2.5]{Martin}}]
		Let us denote $\BF$ the set of curves (recall that they are maps $h:\N\to\Z$ with bounded variation).
		\\
		For $h,h'\in\BF$, we put $h\preceq h'$ if there exists $M>0$ such that $h(t)\leq h'(t)+M$ for all $t\in \N$.
		We put $h\prec h'$ if, besides $h'\not\preceq h$.
		$\preceq$ is a preorder relation on $\BF$, and we note $\sim$ the corresponding equivalence relation.
		\\
		We also note 
		$h\pprec h'$ if $\lim_{t\to\infty}h'(t)-h(t)=+\infty$.
		$\pprec$ is a transitive relation which is finer than $\prec$. 
		\\
		The preorder $\preceq$ induces a notion of (closed, open, semi-open) curve \dfn{interval}, with some bounds $h'\preceq h''$, noted $[h',h'']$, $]h',h''[$, $[h',h''[$, $]h',h'']$.
		We say that the interval is \dfn{nondegenerate} if $h'\prec h''$.
		For an interval $S\subseteq\BF$ with bounds $h'$ and $h''$, we also note $\oin{S}=\set h\BF{h'\pprec h\pprec h''}\subset]h',h''[$.
	\end{defi}
	A \dfn{direction} will implicitly refer to an equivalence class for $\sim$ (sometimes abusively confused with one representative).
	It is not so hard to get convinced that equicontinuity properties are preserved by $\sim$.
	\begin{rmq}\label{r:eqdir}
		Let $F$ be a CA over $A^\Z$ and $h,h'\in\BF$. If $h\sim h'$, then $\E_{F,h}=\E_{F,h'}$.
	\end{rmq}
	In particular, $F$ is almost equicontinuous\resp{equicontinuous} along $h$ if and only if $F$ is almost equicontinuous\resp{equicontinuous} along $h'$.
	\begin{proof}
		By assumption, there exists $M\in\N$ such that $-M+h(x)\leq h'(t)\leq h(t)+M, \forall t\in\N$.
		Let $l\in\N$ and $x\in A^\Z$.
		Let us show that if $x\in A^\Z$ is $2^{-M-l}$-stable along $h$, then it is $2^{-l}$-stable along $h'$.
		So, assume that there exists $k\in\N$ such that $\forall y\in A^\Z, x_{\cc{-k}{k}}=y_{\cc{-k}{k}}\implies\forall t\in\N, F^t\sigma^{h(t)}(x)_{\cc{-M-l}{M+l}}=F^t\sigma^{h(t)}(x)_{\cc{-M-l}{M+l}}~.$
		In other words,
		\[\forall y\in A^\Z, x_{\cc{-k}{k}}=y_{\cc{-k}{k}}\implies\forall t\in\N, F^t(x)_{\cc{-M-l+h(x)}{M+l+h(x)}}=F^t(x)_{\cc{-M-l+h(x)}{M+l+h(x)}}.\]
		By assumption, we get that $\cc{-l+h'(t)}{l+h'(t)}\subseteq\cc{-M-l+h(t)}{M+l+h(t)}, \forall t\in\N.$ Thus, \[\forall y\in A^\Z, x_{\cc{-k}{k}}=y_{\cc{-k}{k}}\implies\forall t\in\N, F^t\sigma^{h'(t)}(x)_{\cc{-l}{l}}=F^t\sigma^{h'(t)}(x)_{\cc{-l}{l}}~,\]
		which is exactly $2^{-l}$-stability of $x$.
		{\\
			Hence, if $x$ is an equicontinuity point along $h$, then $x$ is an equicontinuity point along $h'$. The converse is symmetric.}
	\end{proof}
	
	\section{Limit sets, asymptotic sets, realms}\label{s:attractor}
	We will define notions that deal with asymptotic behavior of a DS $\FF=(F_t)_t$.
	
	\subsection{Limit sets}\label{s:lim}
	The ($\Omega$-) \dfn{limit set} of $U\subseteq X$ is the set $\Omega_\FF(U)=\bigcap_{T\in\N}\cl{\bigcup_{t\ge T}F_t(U)}$, and the \dfn{asymptotic set} of $U\subseteq X$ is the set $\omega_\FF(U)=\bigcup_{x\in U}\Omega_\FF(\{x\})$.
	By compactness, these sets are nonempty (decreasing intersection of nonempty closed subsets).
	$\Omega_\FF(U)$ is compact, but $\omega_\FF(U)$ may not be, even for $U=X$ (see Example~\ref{e:lgliders}).
	Remark that $\Omega_\FF(U)\supseteq\bigcap_{t\in\N}F_t(U)$, and this is an equality if $U$ is a closed $\FF$-invariant set.
	
	We note $\Omega_\FF=\Omega_\FF(X)$ and $\omega_\FF=\omega_\FF(X)$.
	For more about the asymptotic set of dynamical systems, one can refer to \cite{ultfini}. Note that it was called \emph{accessible} set in \cite{Durand}, and \emph{ultimate} set in \cite{nilpeng}.
	
	The following remark easily follows from the compactness of $X$, and can be understood as the fact that every set which is at positive distance from $\Omega_\FF$ is transient, that is, ultimately does not appear.
	\begin{rmq}\label{r:omegastable}
		For every $U$, $\lim_{t\to\infty}d(F_t(U),\Omega_\FF(U))=0$. 
	\end{rmq}
	
	In the uniform case, it is clear that asymptotic sets are invariant.
	Here is a generalization of this fact.
	\begin{prop}\label{p:liminv}
		Let $\FF=(G_{\cc1t})_t$ be a semi-uniform DS,
		$U\subseteq X$, and $j\in\N$.
		\begin{enumerate}
			\item If $y\in\Omega_\FF(U)$, then $(G_{t+\cc1j}(y))_t$ admits a limit point in $\Omega_\FF(U)$.
			\item Conversely, if $z\in\Omega_\FF(U)$, then it is a limit point of $(G_{t+\cc1j}(y))_t$ 
			for some $y\in\Omega_\FF(U)$.
	\end{enumerate} \end{prop}
	Of course, the corresponding statement is also true for the $\omega$, which is defined as a union of limit sets.
	\begin{proof}~\begin{enumerate}
			\item 
			By assumption, there are increasing times $(t_k)_{k\in\N}$ and points $(x_k)_{k\in\N}\in U^\N$ such that $\lim_{k\to\infty}F_{t_k}(x_k)=y$.
			Let $\varepsilon>0$.
			By equicontinuity of the $(G_{t+\cc1j})_t$, there exists $\delta>0$ such that for all $z,z'$ with $d(z,z')<\delta$, we have $\forall t\in\N,d(G_{t+\cc1j}(z),G_{t+\cc1j}(z'))<\varepsilon/2$.
			Then there is $K\in\N$ such that for all $k\ge K$, $d(F_{t_k}(x_k),y)<\delta$, so that $d(F_{t_k+j}(x_k),G_{t_k+\cc1j}(y))<\varepsilon/2$.
			If $z$ is a limit point for $(G_{t_k+\cc1j}(y))_{k\in\N}$, one sees that there exist infinitely many $k$ such that 
			$d(F_{t_k+j}(x_k),z)<\varepsilon$, so that $z$ is also in $\Omega_\FF(\sett{x_k}{k\in\N})\subseteq\Omega_\FF(U)$.
			\item Now let $z\in\Omega_\FF(U)$, so that it is the limit point of $(F_{t_k}(x_k))_{k\in\N}$ for some $(x_k)_{k\in\N}\in U^\N$ and $(t_k)$ an increasing sequence, that we can assume to be greater than $j$.
			By compactness, $(F_{t_k-j}(x_k))_{k\in\N}$ admits a limit point $y\in\Omega_\FF(U)$.
			By triangular inequality, we have 
			{$d(G_{t_k-j+\cc1j}(y),z)\le d(G_{t_k-j+\cc1j}(y),F_{t_k}(x_k))+d(F_{t_k}(x_k),z)$}.
			When $k$ goes to $\infty$, the second term of the sum converges to $0$, and a subsequence of the first term converges to $0$, thanks to equicontinuity of 
			{$(G_{t_k-j+\cc1j})_{k\in\N}$}.
			\popQED\end{enumerate}\end{proof}
	The following corollary is useless for the purpose of the present paper, but may help the reader to connect with the known uniform case.
	\begin{crl}
		If $G$ is a uniform DS and $U\subseteq X$, then $G(\Omega_G(U))=\Omega_G(U)$.
	\end{crl}
	\begin{proof}
		$G_{t+\cc11}=G_t=G$ for every $t\in\N$, so each point of Proposition~\ref{p:liminv} gives one inclusion.
	\end{proof}
	
	
	\subsection{Realms}
	
	The \dfn{realm} (of attraction) of $V$ is:
	\[\AN[\FF]V=\set xX{\omega_\FF(x)\subseteq V}~.\]
	The realm is sometimes called the basin of attraction; we prefer another name to recall that it is relevant even for sets $V$ which have no \emph{attractive} property.
	
	The \dfn{direct realm} of $V$ is the set $\FN[\FF]V=\bigcup_{T\in\N}\bigcap_{t\ge T}F^{-1}_t(V)$ of configurations whose orbits lie ultimately in $V$. 
	
	From the definition and some compactness arguments, the reader can be convinced of the following remarks.
	Note that the realm and the direct realm are related through opposite inclusions, depending on whether the set is open or closed.
	\begin{rmq}\label{r:union}
		Let $V\subseteq X$, and $V_i\subseteq X$ for any $i$ in some arbitrary set $I$.
		\begin{enumerate}
			\item\label{i:apprx} $\AN[\FF]V\subseteq\bigcap_{\varepsilon>0}\FN[\FF]{\B_\varepsilon(V)}$.
			\item\label{i:app} If $V$ is closed, then $\AN[\FF]V=\set xX{\lim_{t\to\infty}d(F_t(x),V)=0}=\bigcap_{\varepsilon>0}\FN[\FF]{\B_\varepsilon(V)}\supseteq\FN[\FF]V$.
			\item On the contrary, if $V$ is open, then $\AN[\FF]V\subseteq\FN[\FF]V$.
			\item\label{i:union} $\AN[\FF]{\bigcup_{i\in I} V_i}\supseteq\bigcup_{i\in I}\AN[\FF]{V_i}$.
			\item\label{i:inter} $\AN[\FF]{\bigcap_{i\in I} V_i}=\bigcap_{i\in I}\AN[\FF]{V_i}$. 
	\end{enumerate}\end{rmq}
	Conjugating the realm operator with complementation is also very relevant dynamically: as stated in the following remark.
	\begin{rmq}\label{r:compl}
		For every DS $\FF$ and subset $V$, the set of points whose orbits have a limit point in $V$ is $\compl{\AN[\FF]{\compl V}}=\set xX{\omega_\FF(x)\cap V\ne\emptyset}\supseteq\AN[\FF]V$, and the set of points whose orbits visit $V$ infinitely many times is $\compl{\FN[\FF]{\compl V}}=\bigcap_{T\in\N}\bigcup_{t\ge T}F_t^{-1}(V)\supseteq\FN[\FF]V$.
		\\
		Note that $\compl{\AN[\FF]{\compl V}}$ is nonempty if and only if $V$ intersects $\omega_\FF$.
		\\
		From Remark~\ref{r:union}, if $V$ is closed\resp{open}, then $\compl{\AN[\FF]{\compl V}}$ includes\resp{is included in} $\compl{\FN[\FF]{\compl V}}$.
	\end{rmq}
	
	\subsection{Related concepts}\label{ss:nwand}
	
	\paragraph{Nonwanderingness.}
	Let $\FF$ be a DS over space $X$.
	We say that $\FF$ is \dfn{nonwandering} if for every nonempty open set $U\subseteq X$, $\compl{\FN[\FF]{\compl U}}$ is not meager. 
	
	This definition does not give a specific role to time $0$, unlike, seemingly, the classical definition, for uniform DS. Nevertheless, they are equivalent in the uniform case, which helps understand the essence of that concept.
	\begin{rmq}
		Let $F$ be a uniform DS.
		Then $F$ is nonwandering if and only if, for every nonempty open set $V\subseteq X$, there exists $t\ge1$ such that $V\cap F^{-t}(V)\ne\emptyset$.
	\end{rmq}
	\begin{proof}~\begin{itemize}
			\item If $(F^t)$ is nonwandering, then the set $\compl{\FN[F]{\compl V}}$ of point whose orbits visit $V$ infinitely many times is in particular nonempty.
			Let $x$ be such a point, and $t_1<t_2$ be two time steps such that $y=F^{t_1}(x)$ and $F^{t_2-t_1}(y)=F^{t_2}(x)$ are both in $V$.
			It is then clear that {$y\in V\cap F^{t_1-t_2}(V)$.}
			\item Now suppose that for every nonempty open set $U\subseteq X$, there exists $t\ge1$ such that $U\cap F^{-t}(U)\ne\emptyset$.
			Let us show by induction on $n\in\N$ that there exist distinct time steps $0=t_0<t_1<\cdots<t_n$ at which the set $W_{(t_0,t_1,\cdots,t_n)}(U)=\bigcap_{i=0}^nF^{-t_i}(U)$ of points whose orbits visit $U$ is nonempty.
			It is trivial for $n=0$. Suppose that $W_{(0,t_1,\cdots,t_n)}(U)\ne\emptyset$. By assumption, we have that there exists $t\ge1$ such that $W_{(0,t_1,\cdots,t_n)}(U)\cap F^{-t}(W_{(0,t_1,\cdots,t_n)}(U))\ne\emptyset$.
			In particular, $W_{(0,t,t+t_1,t+t_2,\cdots,t+t_n)}(U)$ contains this intersection, so that it is nonempty.
			
			Now consider the set \[W_n(V)=\bigcup_{\begin{subarray}c(t_1,\cdots,t_n)\in\N^{n}\\0<t_1<\cdots<t_n\end{subarray}}W_{(0,t_1,\cdots,t_n)}(V)\] of points of $U$ whose orbits visit $U$ at least $n$ more times.
			Note that it is an open set, which is dense in $V$ because it includes the nonempty $W_n(U)\subseteq U$, for every open subset $U\subseteq V$.
			The set $\compl{\FN[F]{\compl V}}$ of points whose orbits visit $V$ infinitely many times can be written as the intersection $\bigcap_{n\in\N}W_n(V)$, and is thus comeager in $V$.
			\popQED\end{itemize}\end{proof}
	
	Moreover, let us mention, even if it will not be used later, that uniform DS are known to admit a nonempty largest nonwandering subsystem, containing, as a comeager set, the set of \dfn{recurrent} points, which are those points $x\in X$ such that $x\in\omega_F(x)$ (see for instance \cite{Vries}).
	Clearly, the set of recurrent points is a subset of the asymptotic set.
	
	\paragraph{Nilpotence.}
	We say that $\FF$ is \dfn{nilpotent} if there is a point $z\in X$ such that $\exists T\in \N, \forall x\in X,\forall t\ge T,F_t(x)=z$.
	$\FF$ is \dfn{asymptotically nilpotent} if $\omega_\FF$ is a singleton.
	
	It is known that CA are nilpotent if and only if their limit set is finite (see for instance \cite{Culik}).
	Also, it has been shown \cite{ultfini} that asymptotically nilpotent CA are actually nilpotent.
	In that case (see for instance \cite{Culik,ultfini}), 
	$z=\sigma(z)$, so that the CA is actually nilpotent in every direction.
	
	
	\paragraph{Asymptotic pairs.}
	Two points $x,y\in X$ are said to be \dfn{asymptotic} 
	to each other (or $(x,y)$ is an \dfn{asymptotic pair}) whenever $\lim_{t\to\infty} d(F_t(x),F_t(y))=0.~$ 
	The \dfn{asymptotic class} of $y$ is the set $\asym[\FF]{y}$ of points asymptotic to it. 
	Let us generalize the realm notations to every sequence $(V_t)_{t\in\N}$ of closed subsets of $X$, 
	by defining: $\AN[\FF]{(V_t)_t}=\set xX{\lim_{t\to\infty}d(F_t(x),V_t)=0}$ and $\FN[\FF]{(V_t)_t}=\set xX{\exists T\in\N,\forall t\ge T,F_t(x)\in V_t}$.
	We may also note $\AN[\FF]{(y_t)_t}$ and $\FN[\FF]{(y_t)_t}$ if $V_t$ is a singleton $\{y_t\}$.
	With this notation,  $\asym[\FF]{y}=\AN[\FF]{(\{F_t(y)\})_t}$.
	
	One can observe from the definition that $y\in\asym[\FF]y\subseteq 
	\AN[\FF]{\omega_\FF(y)}$.
	
	The following remark states that, in a finite space, asymptotic pairs correspond to ultimately equal orbits.	
	\begin{rmq}\label{r:fini}
		Let $G$ be a uniform DS over a finite space $X$, and $x,y\in X$.
		If $x$ and $y$ are asymptotic, 
		then $\exists t\in\N,G^t(x)=G^t(y)$.
		In particular, if $G$ is injective (or surjective), then $x=y$.
	\end{rmq}
	\begin{proof} 
		The first statement comes from $X$ being discrete.
		The second statement is clear because if $X$ is finite, then injectivity or surjectivity of $G$ are equivalent to bijectivity of any $G^t$.
	\end{proof}
	
	The following remark states that when an asymptotic class is big, then it should contain many equicontinuous points.
	\begin{rmq}\label{r:asymeq}
		If $\asym[\FF]{y}$ is comeager\resp{not meager} in some nonempty open subset $U\subseteq X$, for some $y\in X$, then $\E_\FF$ is comeager\resp{not meager} in $U$.
	\end{rmq}
	In particular, note that $\E_\FF\cap\asym[\FF]y$ is also comeager\resp{not meager} in $U$.
	\begin{proof}
		The assumption gives that for every $n\ge 1$, the union $\bigcup_{T\in\N}\bigcap_{t\ge T}F_t^{-1}(\cl\B_{1/n}(F_t(y)))$ of closed sets is comeager in $U$, as a superset of $\asym[\FF]y$.
		{Hence, $\bigcup_{T\in\N}\bigcap_{t\ge T}F_t^{-1}(\cl\B_{1/n}(F_t(y)))$ is not meager in any nonempty open subset $V\subseteq U$. This implies that there is $T\in\N$ such that $\bigcap_{t\ge T}F_t^{-1}(\cl\B_{1/n}(F_t(y)))$ is not meager in $V$; as a closed set, and by Remark~\ref{r:Borel}, it must then include a nonempty open subset $W\subseteq V$.}
		For every $x\in W$, by openness, there exists $\delta>0$ such that for every $z\in\B_\delta(x)$, $z\in W$, which implies that $\forall t\ge T,F_t(z)\in\cl\B_{1/n}(F_t(y))$. In particular, $F_t(x)\in\cl\B_{1/n}(F_t(y))$ and, by triangular inequality, we get $F_t(z)\in\cl\B_{2/n}(F_t(x))\subseteq\B_{3/n}(F_t(x))$.
		We deduce that $x$ is $3/n$-stable.
		In other words, the set of $3/n$-stable points includes nonempty open subsets of every nonempty open subset of $U$. This means that this set is comeager in $U$ for every $n\in\N$. 
		We conclude by Remark~\ref{r:eqstab}.
		\\
		The statement about non-meagerness can be obtained from the other one thanks to Remark~\ref{r:Borel}.
	\end{proof}
	
	\subsection{Decomposition of realms}

	The following proposition can be compared partly to \cite[Lem~3]{milnor}: if a set is decomposable into invariant components, then its realm can be decomposed accordingly.
	\begin{prop}\label{p:decom}
		Let $\FF=(G_{\cc1t})_{t\in\N}$ be a semi-uniform DS.
		Suppose $(V_i)_i$ is a finite collection of closed pairwise disjoint sets which are invariant by every $G_t$.
		Then $\AN[\FF]{\bigsqcup_i V_i}=\bigsqcup_i\AN[\FF]{V_i}$.
	\end{prop}
	\begin{proof}
		Since the $V_i$ are closed and pairwise disjoint, there are at positive pairwise distance. Let $\varepsilon=\min_{i\ne j}d(V_i,V_j)/2>0$.
		By equicontinuity of $(G_t)$, there exists $\delta>0$ such that $\forall t\in\N,\forall x,y\in X,d(x,y)<\delta\implies d(G_t(x),G_t(y))<\varepsilon$.
		Let $x\in\AN[\FF]{\bigsqcup_iV_i}$, so that there exists $T\in\N$ such that $\forall t\ge T,d(F_t(x),\bigsqcup_iV_i)<\min(\delta,\varepsilon)$.
		In particular, there exists $i$ such that $d(F_T(x),V_{i})<\min(\delta,\varepsilon)$.
		Let us show by induction on $t\ge T$ that $d(F_t(x),V_i)<\min(\delta,\varepsilon)$.
		Since $G_{t+1}(V_i)\subseteq V_i$, we have $d(F_{t+1}(x),V_{i})\le d(F_{t+1}(x),G_{t+1}(V_{i}))$.
		They are less than $\varepsilon$ by equicontinuity of $(G_t)$, using the recurrence hypothesis.
		By definition of $\varepsilon$, we have $\min_{j\ne i}d(F_{t+1}(x),V_j)\ge\min_{j\ne i}(d(V_j,V_i)-d(F_{t+1}(x),V_i))\ge\varepsilon$.
		So $\min(\delta,\varepsilon)\ge d(F_{t+1}(x),\bigsqcup_j V_j)=\min_j(d(F_{t+1}(x),V_j)\ge\min(\varepsilon,d(F_{t+1}(x),V_i))$.
		It results that $d(F_{t+1}(x),V_i)\le\min(\delta,\varepsilon)$, as wanted.
		\\
		Since for every $t\ge T$ and $j\ne i$, $d(F_t(x),V_j)\ge\varepsilon$, we deduce $d(F_t(x),V_{i})=\min_jd(F_t(x),V_j)=d(F_t(x),\bigsqcup_jV_j)$ converges to $0$.
		\\
		The other inclusion comes from Point~\ref{i:union} of Remark~\ref{r:union}.
	\end{proof}
	
	\paragraph{Realms of finite sets.}
	Proposition~\ref{p:fini} shows that the realm of a finite set contains a finite number of asymptotic classes.
	We shall use this 
	to show Proposition~\ref{p:sensinfini}.
	It uses the following lemma.
	\begin{lem}\label{l:fini}
		Let $\FF=(G_{\cc1t})_{t\in\N}$ be a semi-uniform DS over space $X$ and $V\subseteq X$ {be finite}.
		Then there exists $\delta>0$ such that for all $x,x'\in\AN[\FF]V$ and $T\in\N$ with $d(F_T(x),F_T(x'))\le\delta$, and $\forall t\ge T,d(F_t(x),V)\le\delta$ and $d(F_t(x'),V)\le\delta$, $(x,x')$ is an asymptotic pair.
	\end{lem}
	\begin{prop}\label{p:fini}
		Let $\FF=(G_{\cc1t})_{t\in\N}$ be a semi-uniform DS over space $X$ and $V\subseteq X$ be finite. 
		Then there are at most $\card V$ asymptotic classes in $\AN[\FF]V$.
	\end{prop}
	\begin{proof}
		For $0\le i\le\card V$, let $x_i\in\AN[\FF]V$, and $\delta$ be as in Lemma~\ref{l:fini}.
		There exists $T\in\N$ such that for all $i$, $\forall t\ge T,d(F_t(x_i),V)<\delta/2$, and in particular, $\exists y_i\in V,d(F_T(x_i),y_i)<\delta/2$.
		By the pigeon-hole principle, there are distinct $i,j$ such that $y_i=y_j$, so that $d(F_T(x_i),F_T(x_j))<\delta$ by the triangular inequality.
		By Lemma~\ref{l:fini}, we then know that $(x_i,x_j)$ is an asymptotic pair.
		Hence we can partition $\AN[\FF]V$ into at most $\card V$ asymptotic classes.
	\end{proof}
	\begin{proof}[Proof of Lemma~\ref{l:fini}]
		Let $\varepsilon=\frac13\min\sett{d(y,y')}{y,y'\in V,y\ne y'}>0$.
		By equicontinuity of $(G_t)$, there exists $\delta>0$ such that $\forall t\in\N,\forall x,x'\in X,d(x,x')\le\delta\implies d(G_t(x),G_t(x'))<\varepsilon$.
		Without loss of generality, we can assume $\delta\le\varepsilon$.
		Let $x,x'$ be as in the statement of the lemma, and for $t\in\N$, let $y(t)\in V$ be such that $d(F_t(x),y(t))=d(F_t(x),V)$, and $y'(t)$ be defined similarly.
		Let us show by induction on $t\ge T$ that $y(t)=y'(t)$, which by definition of $\varepsilon$, is equivalent to $d(y(t),y'(t))<3\varepsilon$.
		First, by the triangular inequality, $d(y(T),y'(T))\le d(y(T),F_T(x))+d(F_T(x),F_T(x'))+d(F_T(x'),y'(T))<3\delta\le3\varepsilon$.
		\\
		Now suppose this is true for $t\ge T$, and let us prove it for $t+1$.
		By the triangular inequality, we also have $d(y(t+1),y'(t+1))\le d(y(t+1),F_{t+1}(x))+d(G_{t+1}F_t(x),G_{t+1}F_t(x'))+d(F_{t+1}(x'),y'(t+1))$.
		The first and third terms are at most $\delta$ by hypothesis, while the central one is at most $\varepsilon$ by definition of $\delta$.
		All in all, we get that $y(t+1)=y'(t+1)$.
		We can conclude with, once again, the triangular inequality: $d(F_t(x),F_t(x'))\le d(F_t(x),y(t))+d(y(t),y'(t))+d(y'(t),F_t(x'))$.
		If $t\ge T$, this is $d(F_t(x),V)+0+d(V,F_t(x'))\to_{t\to\infty}0$.
	\end{proof}
	
	\subsection{Realms for cellular automata}
	The following proposition is very important to show Proposition~\ref{p:alldir}: transitivity of the shift brings some properties to realms and direct realms of shift-invariant sets through CA.
	\begin{prop}\label{p:intnvdense}Let $\FF=(F_t)_{t\in\N}$ be a sequence of CA over $X=A^\Z$, and $V\subseteq A^\Z$.
		\begin{enumerate}
			\item\label{i:shiftinv} $\omega_{F,h}(\sigma(V))=\sigma(\omega_{F,h}(V))$ and $\AN[\FF]{\sigma(V)}=\sigma(\AN[\FF]V)$.
			\item\label{i:intnvdense} If $V$ is strongly $\sigma$-invariant, then $\AN[\FF]V$ either has empty interior or includes a dense open set; it is either nowhere dense or dense. If, moreover, $V$ is closed, then $\AN[\FF]V$ is either comeager or meager.
			\item\label{i:densdom} If $V$ is a $2k+1$-SFT and $\AN[\FF]V$ has nonempty interior, then $\FN[\FF]V$ is dense.
			\item If $V$ is a subshift, then $\FN[\FF]V$ is meager, unless $F^{-1}_T(V)$ is full, for some $T\in\N$.
	\end{enumerate}\end{prop}
	\begin{proof}~\begin{enumerate}
			\item This is clear by definition that $\omega_\FF(\sigma(x))=\sigma(\omega_\FF(x))$.
			\item 
			From the previous point, $\AN[\FF]V$ is strongly $\sigma$-invariant.
			Besides, one can see that, if $V$ is closed, then $\AN[\FF]V=\bigcap_{\varepsilon>0}\FN[\FF]{\B_\varepsilon(V)}$ has the Baire property.
			The three statements then come from Lemma~\ref{l:comeager} (applied to the uniform DS $\sigma$).
			\item Let us show that, for an arbitrary $w\in A^*$, $[w]\cap\FN[\FF]V$ is nonempty.
			Since $\AN[\FF]V$ has nonempty interior, there exists $u\in A^*$ such that $[u]\subseteq\AN[\FF]V$.
			\\
			Let $x=\dinf{(uw)}\in[u]$ be the periodic configuration of period $p=\len{uw}$ and such that $x_{\co0p}=uw$.
			\\
			Since $x\in[u]\subseteq\AN[\FF]V$, there exists $T\in \N$ such that $\forall t>T, d(F_t(x),V)<2^{-k-p}$.
			Since $\forall n\in\Z,x=\sigma^{np}(x)$, we even have:
			\[\forall t>T, \forall n\in\Z,d(F_t\sigma^{np}(x),V)<2^{-k-p}~.\]
			By Point \ref{i:lipschmod} of Lemma~\ref{l:lipschmodSFT}, for all such $t>T$,
			$\forall i\in \Z, \sigma^iF_t(x)\in \B_{2^{-k}}(V)$.
			Since $V$ is a $2k+1$-SFT, $F_t(x)\in V$, by Point \ref{i:SFT} of Lemma~\ref{l:lipschmodSFT}, that is, $x\in\FN[\FF]V$.
			By shift-invariance, we also have that $\sigma^{\len u}(x)\in[w]\cap\FN[\FF]V$.
			\item By definition, $\FN[\FF]V\subseteq\bigcup_{T\in\N}{F^{-1}_T(V)}$.
			If for every $T\in\N$, $F^{-1}_T(V)$ is not full, Point~\ref{i:densopen} of Lemma~\ref{l:comeager} gives that it has empty interior (because it is closed and strongly $\sigma$-invariant).
			In the end, $\FN[\FF]V$ is meager.
			\popQED\end{enumerate}
	\end{proof}
	Unsurprisingly, realms of CA behave well with respect to the shift.
	\begin{prop}\label{p:omegainv}
		Let $F$ be a CA, $h$ a curve, and $U$ be strongly $\sigma$-invariant.
		Then $F(\omega_{F,h}(U))=\sigma(\omega_{F,h}(U))=\omega_{F,h}(U)$.
	\end{prop}
	\begin{proof}
		By Point~\ref{i:shiftinv} of Proposition~\ref{p:intnvdense}, 
		$\omega_{F,h}(U)$ is strongly $\sigma$-invariant, and since $F$ commutes with $\sigma$, we have that $\forall t\ge1,F\sigma^{h(t)-h(t-1)}(\omega_{F,h}(U))=F(\omega_{F,h}(U))$.
		By Proposition~\ref{p:liminv} (applied to $j=1$ and $G_t=F\sigma^{h(t)-h(t-1)}$, so that $\sett{G_t}{t\in\N}$ is finite), we obtain that for any $y,z\in\omega_{F,h}(U)$, $F\sigma^k(y)\in\omega_{F,h}(U)$ for some $k$, and $z=F\sigma^l(x)$ for some $l$ and some $x\in\omega_{F,h}(U)$.
	\end{proof}
	
	\begin{crl}\label{c:fini}
		Let $F$ be a CA, $h$ a curve and $U$ such that $V=\omega_{F,h}(U)$ is finite 
		and $U=\AN[F,h]V$ {is strongly $\sigma$-invariant.}
		Then $F$ induces a self-bijection of $V$, and $U=\bigsqcup_{y\in V}\asym[F,h]y$.
	\end{crl}
	\begin{proof}
		By Proposition~\ref{p:omegainv}, we see that $F(V)=V$, so that $F$ induces a surjection, hence a bijection of $V$.
		\\
		By Proposition~\ref{p:fini}, there are at most $\card V$ asymptotic classes in $U$.
		By the first point, $V\subseteq U$, so that each $y\in V$ should be in one of these classes.
		By Remark~\ref{r:fini}, they are all in distinct classes, so that we obtain 
		{$U\supseteq\bigsqcup_{y\in V}\asym[F,h]y$} (the converse inclusion being trivial).
	\end{proof}
	
	\paragraph{Attractors.}
	In a DS, an \dfn{attractor} is the limit set of an \dfn{inward set}, that is an open set $U$ such that $F_{t+1}(\cl U)\subseteq F_t(U)$, for all $t\in\N$.
	There are other definitions of attractors in the literature, but this one, found for example in \cite{hurley,Petr}, is particularly relevant in totally disconnected spaces, where $U$ can equivalently simply be assumed to be an invariant clopen set.
	References \cite{kurkasubshift,formenti} focus on subshift attractors of CA: in that case the attractor enjoys a definition as the limit set of a so-called \emph{spreading} cylinder.
	$\Omega_\FF=\Omega_\FF(X)$ is then the (unique) {maximal attractor}.
	A \dfn{quasi-attractor} is an intersection of attractors (possibly empty, in our setting).
	The minimal quasi-attractor is thus the intersection of all attractors.
	
	Directly from Point~\ref{i:intnvdense} of Proposition~\ref{p:intnvdense}, we recover the following (recall that every attractor has a nonempty open realm).
	\begin{crl}[\cite{Pietro}]\label{c:subattr}
		For any CA in any direction, the realm of any subshift attractor is a dense open set.
	\end{crl}
	The following example will be described more deeply in Example~\ref{e:sigmin}, but gives here a first illustration of the concept of limit set and attractor.
	\begin{exmp}\label{e:min}
		Let $\Min$ be defined over $\{0,1\}^\Z$ by $\Min(x)_i=\min (x_i,x_{i+1})$.
		One has $\{\dinf0\}=\bigcap_{k\ge 0}V_k$, where $V_k=\Omega_\Min([0]_k)=\sett{x\in\Omega_\Min}{\forall i\leq k, x_i=0}$ is an attractor but not a subshift, for every $k\in\Z$ (see \cite{kurka}).	$\{\dinf0\}$ is the unique minimal quasi-attractor, and its realm $\AN[\Min]{\dinf0}=\bigcap_{k\in\Z}\sett{x\in \{0,1\}^\Z}{\exists i\geq k,x_i=0}$ is comeager.	
	\end{exmp}
	
	This property of having a comeager realm motivates the next definition.
	
	\section{The generic limit set}\label{s:genericsystem}
	Milnor \cite{milnor} suggests the following definition, which is the purpose of the present section.
	
	\begin{defi}
		Being given a DS $\FF$, the \dfn{generic limit set} $\tilde\omega_\FF$ is the intersection of all the closed subsets of $X$ which have a comeager realm of attraction.
	\end{defi}
	The generic limit set $\tilde\omega_\FF$ can actually be defined as the smallest closed subset of $X$ with a comeager realm, thanks to the following proposition.
	\begin{prop}\label{p:exist}
		Let $\FF$ be a DS. The realm of the generic limit set is comeager.
	\end{prop}
	In particular, it is nonempty!
	But much more thant this: it is the smallest closed set which includes all limit points of all \emph{generic} orbits. 
	\begin{proof}
		Any compact metric space admits a countable basis: there exists a countable set $\sett{U_i}{i\in\N}$ of closed subsets such that every closed set $U$ can be written as $\bigcap_{i\in I_U}U_i$ for some $I_U\subseteq\N$.
		In particular, $\tilde\omega_\FF$ is the intersection $\bigcap_U\bigcap_{i\in I_U}U_i$, where $U$ ranges over closed sets with comeager realm; that is, $\tilde\omega_\FF=\bigcap_{i\in I}U_i$, where $I$ is the union of $I_U$, for every closed $U$ with comeager realm.
		If $i$ is in $I$, then it is in some $I_U$, so that $U\subseteq U_i$, where $U$ has comeager realm, so that $U_i$ has comeager realm, too.
		\\
		By Point~\ref{i:inter} of Remark~\ref{r:union}, $\AN[\FF]{\tilde\omega_\FF}=\AN[\FF]{\bigcap_{i\in I}{U_i}}=\bigcap_{i\in I}\AN[\FF]{U_i}$.
		We know that an intersection of countably many comeager sets is comeager. Then $\AN[\FF]{\tilde\omega_\FF}$ is comeager.
	\end{proof}
	
	Note that the generic limit set is the closure of the asymptotic set of some comeager set {(it is exactly the closure of the asymptotic set of its realm)}, but it may not be the asymptotic set of any set: see for example Example~\ref{e:lgliders2}, where the generic limit set is full, but the asymptotic set is not.
	
	\begin{rmq}\label{r:geninter}
		Let $\FF$ be a DS over space $X$ and $V\subseteq X$.
		\begin{enumerate}
			\item\label{i:clos} If $V$ does not intersect $\tilde\omega_\FF$, then $\compl{\AN[\FF]{\compl V}}$ is meager.
			In particular, if $V$ is closed, then $\compl{\FN[\FF]{\compl V}}$ is meager, and there is no nonempty open set $U$ in which $\bigcup_{t\ge T}F_t^{-1}(V)$ is dense for all $T\in\N$.
			\item\label{i:open} If $V$ is open and intersects $\tilde\omega_\FF$, then $\compl{\AN[\FF]{\compl V}}$ is not meager.
			In particular, $\compl{\FN[\FF]{\compl V}}$ is not meager, and there exists a nonempty open set $U$ in which $\bigcup_{t\ge T}F_t^{-1}(V)$ is dense for all $T\in\N$.
		\end{enumerate}
	\end{rmq}
	There are counter-examples when this remark cannot be stated as an equivalence: take for instance open set $V=]0,1[$ in the uniform DS defined by $F_t(x)=x/2^t$ for $x\in[0,1]$.
	\begin{proof}
		We prove the first statement of each point.
		\begin{itemize}
			\item If $\compl V\supseteq\tilde\omega_\FF$, then $\AN[\FF]{\compl V}\supseteq\AN[\FF]{\tilde\omega_\FF}$ is comeager.
			\item If $V$ is open and intersects $\tilde\omega_\FF$, then $\tilde\omega_\FF\setminus V$ is closed; by the minimality of the generic limit set, $\AN[\FF]{\tilde\omega_\FF\setminus V}$ is not comeager.
			This set is equal to $\AN[\FF]{\tilde\omega_\FF}\cap\AN[\FF]{\compl V}$. Since the first one is comeager, we deduce that the second one is not.
		\end{itemize}
		The second statement of each point comes from the inclusions between realm and direct realm in Remark~\ref{r:compl}.
		The third statement comes from Remark~\ref{r:Borel} and the definition of $\compl{\FN[\FF]{\compl V}}$.
	\end{proof}
	A consequence of this is the following proposition: the generic limit set intersects any closed set with dense realm.
	\begin{prop}\label{p:densinter}
		Suppose $\FF$ is a DS and $V$ a closed set with dense realm $\AN[\FF]V$.
		Then $V$ intersects $\tilde\omega_\FF$.
	\end{prop}
	\begin{proof}
		Point~\ref{i:app} of Remark~\ref{r:union} gives that {$\AN[\FF]V=\bigcap_{n>0}\FN[\FF]{\B_{1/n}(V)}$.}
		It results that each $\FN[\FF]{\B_{1/n}(V)}$ is dense.
		Its open superset $\bigcup_{t\geq T}F_t^{-1}(\B_{1/n}(V))$ should also be dense. Hence, $\compl{\FN[\FF]{\compl{\B_{1/n}(V)}}}$ is comeager. In particular, it is not meager.
		Then Point~\ref{i:clos} of Remark~\ref{r:geninter} gives that 
		$\cl{\B_{1/n}(V)}$ intersects $\tilde\omega_\FF$.
		Since this is true for every $n>0$, and $\cl{B_{1/n}(V)}$ is closed, $V$ should intersect $\tilde\omega_\FF$.
	\end{proof}
	
	\subsection{Nonwandering systems}
	We will see that, for nonwandering dynamical systems, the generic limit set is the full space.
	Let us prove a more general result, which will be also useful for Corollary~\ref{c:obliq}.
	We say that a DS $\FF$ is \dfn{semi-nonwandering} if for every open subset $U$ which intersects $\Omega_\FF$, $\compl{\FN[\FF]{\compl U}}$ is not meager. 
	It is clear that a DS over some space $X$ is nonwandering if and only if it is semi-nonwandering and its limit set is $X$ (note that this second property happens, in the uniform case, exactly for surjective systems).

	\begin{prop}\label{p:snwand}
		A DS $\FF$ is semi-nonwandering if and only if $\tilde\omega_\FF=\Omega_\FF$.
	\end{prop}
	\begin{proof}~\begin{itemize}
			\item Suppose $\FF$ is not semi-nonwandering.
			This means that there exists an open set $U$ which intersects $\Omega_\FF$, and such that $\FN[\FF]{\compl U}$ is comeager.
			Since $\compl U$ is closed, $\AN[\FF]{\compl U}\supseteq\FN[\FF]{\compl U}$ is also comeager.
			By definition, $\tilde\omega_\FF$ is then included in $\compl U$, and thus cannot include $\Omega_\FF$. 
			\item Conversely, suppose that $\FF$ is semi-nonwandering, $x\in\Omega_\FF$ and $\varepsilon>0$.
			By definition, $\compl{\FN[\FF]{\compl{\B_\varepsilon(x)}}}$ is not meager.
			By inclusion, neither is $\compl{\FN[\FF]{\compl{\cl\B_\varepsilon(x)}}}$.
			By Point~\ref{i:clos} of Remark~\ref{r:geninter}, we deduce that $\cl\B_\varepsilon(x)$ intersects $\tilde\omega_\FF$.
			Since this is true for every $\varepsilon>0$, we get that $x\in\tilde\omega_\FF$.
			The inclusions $\tilde\omega_\FF\subseteq\cl{\omega_\FF}\subseteq\Omega_\FF$ are always true.
			\popQED\end{itemize}
	\end{proof}
	The following is a direct corollary of Proposition~\ref{p:snwand}.
	\begin{crl}\label{c:wandering}
		A DS $\FF$ over some space $X$ is nonwandering if and only if $\tilde\omega_\FF=X$.
	\end{crl}
	It is clear that the two properties in Corollary~\ref{c:wandering} imply surjectivity.
	Since it is known that surjective CA are all nonwandering (see for instance \cite[Prop~5.23]{kurka}), we get the following.
	\begin{crl}\label{c:surjective}
		A CA $F$ over $A^\Z$ is surjective if and only if $\tilde\omega_F= A^{\Z}$,
		if and only if $\omega_F$ is comeager.
	\end{crl}
	The second statement is also true for uniform DS \cite[Cor~26]{ultfini}.
	\begin{proof}
		The first statement is a direct corollary of Corollary~\ref{c:wandering}.
		
		For the second statement, nonwandering uniform DS are known to admit a comeager set of recurrent points, which are all in $\omega_\FF$ (see Subsection~\ref{ss:nwand}).
	\end{proof}
	
	\begin{figure}[ht]
		\centering
		\includegraphics[width=6 cm]{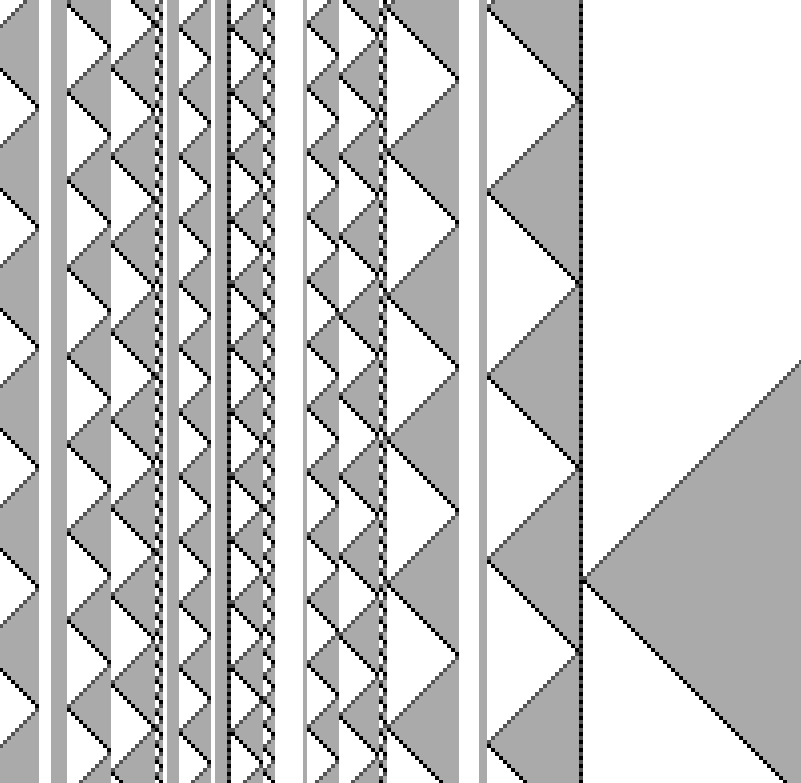}
		\caption{Lonely Gliders:
			space is horizontal and time goes upwards;
			$\leftarrow$ \resp{$>$} are represented by black\resp{white} squares, and $\rightarrow$ \resp{$<$} are represented by dark\resp{light} grey squares.}
		\label{fig:glid}
	\end{figure}
	The following example answers a question left open in \cite{ultfini}: the asymptotic set of a surjective CA is comeager, but not always full.
	\begin{exmp}[Lonely gliders]\label{e:lgliders}
		Let $A=\{>,<,\rightarrow,\leftarrow\}$, and $F$ the CA, defined by the following local rule:
		\[f:(x_{-1},x_0,x_1)\mapsto\soit{
			\rightarrow&\si x_{-1}=\rightarrow\et x_0=<\\
			\rightarrow&\si x_{-1}\ne> \et x_0=\leftarrow\\
			<&\si x_{-1}=> \et x_0=\leftarrow\\
			>&\si x_0=\rightarrow\et x_1=<\\
			\leftarrow&\si x_0=\rightarrow\et x_1\ne<\\
			\leftarrow&\si x_0=> \et x_1=\leftarrow\\
			x_0&\sinon~.}\]
		A typical space-time diagram of this CA is shown in Figure~\ref{fig:glid}.
		Intuitively, each configuration can be decomposed into valid zones, which contain at most one arrow, towards with chevrons $<$ and $>$ are supposed to point. The arrow moves in the direction to which it points, until it reaches the end of the zone (noticed by an \emph{invalid} pattern of the form $ab$, where $a\ne>$ and $b\in\{>,\rightarrow,\leftarrow\}$, or symmetric), in which case it turns back.
		With this in mind, it is not difficult to understand that $F$ is reversible (hence surjective) and that any \emph{invalid} pattern is a blocking word.
		From Corollary~\ref{c:surjective}, the asymptotic set is comeager. Yet, it is not full.
		
	\end{exmp}
	\begin{proof}
		Let us prove that some configuration $x$ with an infinite valid zone which contains one arrow cannot be the limit point of any orbit.
		Indeed, any configuration whose orbit comes arbitrarily close to $x$ should also have an infinite valid zone (because the zones are invariant), and hence at most one arrow in it (like the configuration whose orbit is illustrated in Figure~\ref{fig:glid}.
		Any limit point of such an orbit has no arrow in its infinite valid zone (the arrow \emph{goes to infinity}).
	\end{proof}
	
	\subsection{First property for CA}
	It is rather clear that the generic limit set of a uniform DS induces a subsystem. Let us see that this is true also for directional dynamics CA, and that it is also invariant by shift.
	\begin{prop}\label{p:subsh}
		Let $\FF=(F_t)_t$ be a sequence of CA.
		Then $\tilde\omega_\FF$ is a subshift.
		Its realm is strongly $\sigma$-invariant.
	\end{prop}
	\begin{proof}
		By definition, $\tilde\omega_\FF$ is closed.
		Let $U=\AN[\FF]{\tilde\omega_\FF}$.
		Since $\sigma$ is a homeomorphism, $\sigma^k(U)$ is also comeager for all $k\in\Z$.
		Then $W=\bigcap_{k\in\Z}\sigma^k(U)$ is still comeager, as an intersection of countably many comeager sets.
		One has $W\subseteq U$, so that $\cl{\omega_\FF (W)}\subseteq\cl{\omega_\FF(U)}=\tilde\omega_\FF$.
		Conversely, the definition of $\tilde\omega_\FF$ gives that it is included in $\cl{\omega_\FF (W)}$.
		Overall, $\cl{\omega_\FF(W)}=\tilde\omega_\FF$.
		Since $W$ is strongly $\sigma$-invariant, $\tilde\omega_\FF=\cl{\omega_\FF(W)}$ is also strongly $\sigma$-invariant, by Proposition~\ref{p:omegainv}.
	\end{proof}
	Moreover, Corollary~\ref{c:subattr} directly gives that $\tilde\omega_\FF$ is included in all subshift attractors.
	\begin{prop}\label{p:finv}
		Consider the CA $F$ in some direction $h$.
		Then $\tilde\omega_{F,h}$ is an $F$-invariant subshift.
	\end{prop}
	\begin{proof}
		We just apply Proposition~\ref{p:omegainv} to $\tilde\omega_{F,h}=\cl{\omega_{F,h}(\AN[F,h]{\tilde\omega_{F,h}})}$.
	\end{proof}
	
	\subsection{Indecomposability}
	Now we prove that the generic limit set of a cellular automaton is indecomposable in some sense.
	\begin{prop}\label{p:subsystems}
		Let $V=\bigsqcup^{n-1}_{i=0} V_i$, where $n\in\N$ and the $V_i$ are closed subsets which are invariant by some CA $F$ in some direction $h$, and, strongly, by $\sigma^p$, for some $p>0$.
		If $\AN[F,h]{V}$ has nonempty interior \resp{is not meager}, then there exists $i\in\co0n$ such that $\AN[F,h]{V_i}$ is dense \resp{comeager}.
	\end{prop}
	\begin{proof}
		One has $\AN[F,h]{V}=\bigsqcup^{n-1}_{i=0}\AN[F,h]{V_i}$ by Proposition~\ref{p:decom}.
		\\
		By Point~\ref{i:coseq} of Lemma~\ref{l:comeager}, there exists $i\in\co0n$ such that $\AN[F,h]{V_i}$ is dense \resp{comeager}.
	\end{proof}
	\begin{crl}\label{c:indecomp}
		Let $F$ be a CA and $h$ a curve.
		$\tilde\omega_{F,h}$ cannot be decomposed as a disjoint union of non-trivial subshift subsystems (or even non-trivial strongly $\sigma^p$-invariant subsystems, for some $p>0$).
	\end{crl}
	In other words, we can say that $\tilde\omega_{F,h}$ is connected, when considering the dynamical pseudo-metric related to the action of $(F,\sigma)$: $\tilde d(x,y)=\inf_{i,j\in\Z,s,t\in\N}d(F^s\sigma^i(x),F^t\sigma^j(y))$.
	\begin{proof}
		We assume that $\tilde\omega_{F,h}=\bigsqcup_{i=0}^{n-1}V_i$, where the $V_i$ are closed, invariant, strongly $\sigma^p$-invariant sets.
		By Proposition~\ref{p:subsystems}, there exists $i\in\co0n$ such that $\AN[F,h]{V_i}$ is comeager.
		By definition, $\tilde\omega_{F,h}$ is then included in $V_i$, and hence equal.
	\end{proof}
	
	\subsection{Finite generic limit set}
	If the generic limit set is finite, then the DS is almost equicontinuous.
	\begin{prop}\label{p:sensinfini}
		Let $\FF$ be a semi-uniform DS and $V=\omega(U)$ be finite, for some set $U$ which is comeager in some nonempty open subset $W\subseteq X$. Then $\E_\FF$ is comeager in $W$.
	\end{prop}
	This of course implies that $\E_\FF\cap U$ is comeager in $W$.
	
	Before proving the proposition, we immediately deduce the following.
	\begin{crl}\label{c:sensinfini}~\begin{itemize}
			\item If $\tilde\omega_\FF$ is finite, then $\FF$ is almost equicontinuous.
			\item If $\FF$ has no equicontinuous point, then the asymptotic (and limit) sets of all non-meager sets are infinite.
		\end{itemize}
	\end{crl}
	\begin{proof}[Proof of Proposition~\ref{p:sensinfini}]
		According to Proposition~\ref{p:fini}, {$U\subseteq\AN[\FF]V\subseteq\bigsqcup_{y\in I}\asym[\FF]{y}$ } for some finite set $I$.
		If $U$ is comeager in $W$, then so is this union.
		Let $W'\subseteq W$ be nonempty and open.
		Since a finite union of sets meager in $W'$ is meager in $W'$, we deduce that $\asym[\FF]y$ is not meager in $W'$, for some $y\in I$.
		Remark~\ref{r:asymeq} says that $\E_\FF$ is then not meager in $W'$ either.
		We conclude thanks to Remark~\ref{r:Borel}.
	\end{proof}
	
	In the case of cellular automata with finite generic limit set, we can say more.
	\begin{prop}\label{p:periodicorbit}
		Let $F$ be a CA and $h$ a curve, such that $\tilde\omega_{F,h}$ is finite.
		Then $\tilde\omega_{F,h}$ contains one single (periodic) orbit, of a monochrome configuration $y$, and $\asym[F,h]{y}$ is comeager.
	\end{prop}
	Note that the orbit of this monochrome configuration may be nontrivial (see Example~\ref{e:martin}), but still generic configurations are all asymptotic to a single configuration of that orbit (and not to the others). This could seem paradoxical, since it contrasts with the usual, uniform and synchronous aspect of dynamics of CA over the full set $A^\Z$, but here the genericity notion is not at all $F$-invariant.
	\begin{proof} 
		
		By Proposition~\ref{p:finv}, $\tilde\omega_{F,h}$ is a finite $F$-invariant subshift, so that all configurations are periodic (for the shift). Let $p>0$ be a common period: 
		$\tilde\omega_{F,h}$ can be decomposed as $\bigsqcup_{y\in V}\orb_{F,h}(y)$, where $V\subseteq\tilde\omega_{F,h}$ is a set of orbit representatives. 
		We can apply Corollary~\ref{c:fini}: $\AN[F,h]{\tilde\omega_{F,h}}=\bigsqcup_{y\in\tilde\omega_{F,h}}\asym[F,h]y$. 
		Since every $y\in\tilde\omega_{F,h}$ is strongly $\sigma^p$-invariant, we can apply Point~\ref{i:coseq} of Lemma~\ref{l:comeager} (to the uniform DS $\sigma^p$), and get that there is 
		$y\in\tilde\omega_{F,h}$ such that $\asym[F,h]y$ 
		is comeager.
		Since $\AN[F,h]{\orb_{F,h}(y)}\supseteq\asym[F,h]y$ is comeager, we get that the closed $\orb_{F,h}(y)$ is actually $\tilde\omega_{F,h}$.
		Moreover, since $\sigma$ is an automorphism of $F$, $\sigma(\asym[F,h]y)=\asym[F,h]{\sigma(y)}$ 
		is also comeager.
		Then $\asym[F,h]y\cap\asym[F,h]{\sigma(y)}$ is also comeager, and in particular nonempty.
		By transitivity of the asymptoticity relation, $y$ is asymptotic to $\sigma(y)$.
		Since they both lie in the bijective subsystem of $F$ induced over the finite $\tilde\omega_{F,h}$, Remark~\ref{r:fini} gives that $y=\sigma(y)$, which means that $y$ is monochrome.
		%
	\end{proof}
	
	\subsection{Asymptotic set of equicontinuous points}
	We shall show that if the system is almost equicontinuous, then its generic limit set is exactly the closure of the asymptotic set of its set of equicontinuous points.
	
	The following proposition and its corollary show that the set of equicontinuity points is included in all dense realms.
	\begin{prop}\label{p:eqan}
		Let $\FF$ be a DS and $(V_t)_{t\in\N}$ a sequence of {closed} subsets of $X$.
		Then $\E_\FF\cap \cl{\AN[\FF]{(V_t)_t}}\subseteq \AN[\FF]{(V_t)_t}$.
	\end{prop}
	\begin{proof}
		Let $x\in \E_\FF\cap \cl{\AN[\FF]{(V_t)_t}}$, and $\varepsilon>0$.
		Since $x\in \E_\FF$, there exists $\delta>0$ such that 
		\[\forall y\in \B_{\delta}(x), \forall t\in\N, d(F_t(x),F_t(y))<\varepsilon/2~.\]
		Since $x\in\cl{\AN[\FF]{(V_t)_t}}$, there exists $y\in\B_{\delta}(x)\cap\AN[\FF]{(V_t)_t}$. Hence,
		\[\exists T\in \N,  \forall t> T, d(F_t(y), V_t) < \varepsilon/2~.\] 
		For this $T$, $\forall t> T,  d(F_t(x), V_t)<\varepsilon$.
		Since this is true for every $\varepsilon>0$, we get that $x\in\AN[\FF]{(V_t)_t}$.
	\end{proof}
	In particular, for 
	$z\in X$, we have $\asym[\FF]z\supseteq\cl{\asym[\FF]z}\cap\E_\FF$.
	
	\begin{rmq}
		In \cite[Prop~2.74]{kurka}, it is proved that, in a uniform DS, if an attractor is included in the set of equicontinuous points, so is its realm. 
		In particular, we can deduce from Proposition~\ref{p:eqan} the following result: the realm of any 
		subshift attractor {included in the set of equicontinuous points,} is exactly the set of equicontinuity points.
	\end{rmq}

	For almost equicontinuous DS, Proposition~\ref{p:eqan} means that it is enough to prove that some realm is dense to prove that it is comeager.	
	\begin{crl}\label{c:aleq}
		If $\FF$ is an almost equicontinuous DS, then $\tilde\omega_\FF=\cl{\omega_\FF(\E_\FF)}$.
	\end{crl}
	\begin{proof}
		Since $\AN[\FF]{\tilde\omega_\FF}$ is dense, $\E_\FF\subseteq \AN[\FF]{\tilde\omega_\FF}$ by Proposition~\ref{p:eqan}. 
		Hence, $\omega_\FF(\E_\FF)\subseteq\tilde\omega_\FF$. Since $\tilde\omega_\FF$ is closed, $\cl{\omega_\FF(\E_\FF)}\subseteq\tilde\omega_\FF$.
		Conversely, $\tilde\omega_\FF$ is the intersection of all closed subsets with comeager realms, among which $\cl{\omega_\FF(\E_\FF)}$ (whose realm includes the comeager $\E_\FF$).
		So, $\tilde\omega_\FF=\cl{\omega_\FF(\E_\FF)}$.
	\end{proof}
	
	If the system is equicontinuous, then its generic limit set is its limit set.
	\begin{prop}\label{p:equi}
		If $\FF$ is an equicontinuous DS over space $X$, then $\tilde\omega_\FF=\cl{\omega_\FF}=\Omega_\FF$.
	\end{prop}
	\begin{proof}
		Let $y\in\Omega_\FF$ and $\varepsilon>0$.
		We will show that $\B_\varepsilon(y)$ intersects $\omega_\FF$.
		There exists $\delta$ such that for every $x\in \E_\FF=X$ and every $t\in\N$, $F_t(\B_\delta(x))\subseteq\B_{\varepsilon/2}(F_t(x))$.
		By compactness of $X$, there exists a finite $I\subseteq X$ such that $X=\bigcup_{x\in I}\B_\delta(x)$.
		Since $y\in\Omega_\FF$, there is an infinite $J\subseteq\N$, and for all $t\in J$, some $x_t\in X$ such that $F_t(x_t)\in\B_{\varepsilon/2}(y)$.
		By the pigeon-hole principle, there exists $x\in I$ such that $\B_\delta(x)$ contains infinitely many $x_t$ with $t\in J$.
		This means that for infinitely many $t$, $d(F_t(x),y)\le d(F_t(x),F_t(x_t))+d(F_t(x_t),y)\le\varepsilon/2+\varepsilon/2$.
		We conclude that the orbit of $x$ has a limit point $z\in\omega_\FF(x)\cap\B_\varepsilon(y)$.
		This proves that $\omega_\FF(X)$ is dense in $\Omega_\FF$; by Corollary~\ref{c:aleq}, we obtain $\tilde\omega_\FF=\cl{\omega_\FF(X)} =\Omega_\FF$.
	\end{proof}
	Another remark: it is known that a cellular automaton is nilpotent if and only if its limit set is finite. Hence, it is nilpotent if and only if it is equicontinuous and its generic limit set is finite.
	
	\section{Directional dynamics}\label{s:directional}
	In this section, we study cellular automata while varying the directions.
	
	We have already seen in Remark~\ref{r:eqdir} that the equicontinuity properties are preserved by $\sim$.
	This is also the case for asymptotic sets, as stated in the following remark; the limit set and direct realm are even direction-invariant, provided that the considered set is strongly shift-invariant. 
	\begin{rmq}\label{r:omega}
		Let $F$ be a CA, $V,U\subseteq A^\Z$ be strongly $\sigma$-invariant, {$V_t\subseteq A^\Z$ be closed and strongly $\sigma$-invariant,} for $t\in\N$
		and $h,h'\in\BF$.
		\begin{enumerate}
			\item $\FN[F,h]{(V_t)_t}=\FN[F,h']{(V_t)_t}$.
			\item $\Omega_{F,h}(U)=\Omega_{F,h'}(U)$.
			\item If $h\sim h'$, then $\omega_{F,h}(U)=\omega_{F,h'}(U)$.
			\item If $h\sim h'$, then $\AN[F,h]V=\AN[F,h']V$, and $\tilde\omega_{F,h}=\tilde\omega_{F,h'}$.
	\end{enumerate}\end{rmq}
	\begin{proof}~\begin{enumerate}
			\item Let $x\in\FN[F,h]{(V_t)_t}$, that is, there exists a time $T\in\N$ above which for all times $t>T$, $F^t\sigma^{h(t)}(x)\in V_t$.
			By assumption, we get that $F^t\sigma^{h'(t)}(x)=\sigma^{h'(t)-h(t)}F^t\sigma^{h(t)}(x)\in V_t$.
			\item $\Omega_{F,h}(U)=\bigcap_{T\in\N}\cl{\bigcup_{t\ge T}F^t\sigma^{h(t)}(U)}=\bigcap_{T\in\N}\cl{\bigcup_{t\ge T}F^t(U)}=\Omega_F(U)$.
			\item Let $y\in\omega_{F,h}(U)$; there exist $x\in U$ and an increasing sequence $(n_k)_k$ such that $\lim_{k\to\infty}F^{n_k}\sigma^{h(n_k)}(x)=y$. 
			Since $h\sim h'$,  $(h(t)-h'(t))$ is bounded, for $t\in\N$. We deduce that there is a subsequence $(m_k)_k$ of $(n_k)_k$ such that $(h(m_k)-h'(m_k))$ is constant, say equal to $q\in\Z$.
			We deduce that $F^{m_k}\sigma^{h'(m_k)}(\sigma^q(x))=F^{m_k}\sigma^{h(m_k)}(x)$, so that this term converges to $y$ when $k$ goes to infinity.
			Since $U$ is $\sigma$-invariant, it contains $\sigma^q(x)$, so that $y\in\omega_{F,h'}(U)$.
			\item This can be directly derived from the definitions and the previous point.
			
			
			\popQED\end{enumerate}\end{proof}
	
	\subsection{Oblique directions}\label{s:oblique}
	Let $F$ be a CA with memory $r_-\in\Z$ and anticipation $r_+\in\Z$.
	Then for every $t\in\N$, $F^t$ can be defined by a rule of memory $r_-t$ and anticipation $r_+t$.
	But it could be that smaller parameters also fit.
	This motivates the following definition.
	
	For a sequence $(F_t)_{t\in\N}$ of CA, let us denote $\rpm-(t)$ and $\rpm+(t)$ the minimum possible memory and anticipation for $F_t$, and call them the \dfn{iterated memory} and \dfn{iterated anticipation}.
	Formally, 
	\[\rpm-(t)=\sup\set i\Z{\forall x,y\in A^\Z,x_{\ci i}=y_{\ci i}\implies F_t(x)_0=F_t(y)_0}\] and
	\[\rpm+(t)=\inf\set i\Z{\forall x,y\in A^\Z,x_{\ic i}=y_{\ic i}\implies F_t(x)_0=F_t(y)_0}.\]
	The next remark shows that essentially in the case of a unique direction of equicontinuity (up to $\sim$), the iterated memory is equivalent to the iterated anticipation.
	\begin{rmq}\label{r:eqlyapu}~\begin{enumerate}
			\item For every $t\in\N$, $-\infty<\rpm-(t)\le \rpm+(t)<+\infty$, if and only if $F_t$ is not a constant function.
			\item $(F_t\sigma^{h(t)})_{t\in\N}$ is equicontinuous if and only if $\rpm+\preceq-h\preceq \rpm-$ (in particular, if $F_t$ is never constant, $h\sim-\rpm-\sim-\rpm+$).
	\end{enumerate}\end{rmq}
	In the uniform case, some $F_t$ is constant if and only if the CA is nilpotent.
	\begin{proof}
		The first statement is direct from continuity of $F_t$.
		\\
		If $(F_t\sigma^{h(t)})_t$ is equicontinuous, then there exists $r\in\N$ such that \[\forall t\in\N, x,y\in A^\Z, x_{\cc{-r}r}=y_{\cc{-r}r}\implies F_t\sigma^{h(t)}(x)_0=F_t\sigma^{h(t)}(y)_0~. \]
		Since $F$ commutes with $\sigma$, we get $x_{\cc{-r-h(t)}{r-h(t)}}=y_{\cc{-r-h(t)}{r-h(t)}}\implies F_t(x)_0=F_t(y)_0$.
		\\
		We get that $-r-h(t)\le \rpm-(t)$ and $\rpm+(t)\le r-h(t)$.
		\\
		Conversely, assume that there exists $M\in\N$ such that $\forall t\in\N,\rpm+(t)\le M-h(t)$ and $-h(t)\le M+\rpm-(t)$, and let $l\in\N$.
		Then for every $t\in\N$ and $x,y\in A^\Z$ such that $x_{\cc{-l-M}{l+M}}=y_{\cc{-l-M}{l+M}}$, consider the configuration $z\in A^\Z$ such that $z_i=x_i$ for every $i\in\ic{l+M}$ and $z_i=y_i$ for every $i\in\ci{-l-M}$.
		By the assumed inequalities, $x_{\ic{l+\rpm+(t)+h(t)}}=z_{\ic{l+\rpm+(t)+h(t)}}$ and $z_{\ci{-l+\rpm-(t)+h(t)}}=y_{\ci{-l+\rpm-(t)+h(t)}}$, so that $F_t\sigma^{h(t)}(x)_{\cc{-l}l}=F_t\sigma^{h(t)}(z)_{\cc{-l}l}=F_t\sigma^{h(t)}(y)_{\cc{-l}l}$.
		This proves equicontinuity of $(F_t\sigma^{h(t)})_t$.
		\\
		The consequence between parentheses comes from the first point.
	\end{proof}
	
	In the case of a single CA, we have seen that $\rpm-(t)\ge r_-t$ and $\rpm+(t)\le r_+t$. It is known that these sequences will be asymptotically linear, the slopes being called the Lyapunov exponents (see \cite{mudrunner} for a discussion on possible growths for these sequences).
	
	We say that a direction $h$ is \dfn{oblique} for CA sequence $(F_t)_t$ if $h\notin[-\rpm+,-\rpm-]$. 
	We will show that the generic limit set in an oblique direction is equal to the limit set. 
	
	We say that a DS $\FF$ over space $X$ 
	is \dfn{weakly semi-mixing} if for every nonempty open sets $U,V,U',V'$ such that $V$ and $V'$ intersect $\Omega_\FF$ and for any $T\in\N$, there exists $t\ge T$ such that $U\cap F_t^{-1}(V)\ne\emptyset$ and $U'\cap F_t^{-1}(V')\ne\emptyset$.
	This implies that $\FF$ is \dfn{semi-transitive}, which means that for every nonempty open sets $U,V$ such that $V$ intersects $\Omega_\FF$ and for any $T\in\N$, there exists $t\ge T$ such that $U\cap F_t^{-1}(V)\ne\emptyset$.
	\begin{prop}\label{p:transnwand}
		A DS $\FF$ is semi-transitive if and only if for every sequence $(U_i)_{i\ge1}$ of open subsets intersecting $\Omega_\FF$, the set $\bigcup_{t_1<\cdots<t_i<\cdots}\bigcap_{i\ge1}F_{t_i}^{-1}(U_i)$ is comeager.
		In particular, $\FF$ is then semi-nonwandering.
	\end{prop}
	\begin{proof}~\nopagebreak\begin{itemize}
			\item Suppose $\FF$ is semi-transitive, and that $(U_i)_{i\ge1}$ is a sequence of open sets intersecting $\Omega_\FF$.
			Let us show, by induction over $n\in\N$, that the open set $Z_{n}=\bigcup_{t_1<\cdots<t_n}\bigcap_{1\le i\le n}F_{t_i}^{-1}(U_i)$ is dense.
			Let $U_0$ be any nonempty open subset of $X$.
			Let us show, by induction over $n\in\N$, that the open set $Z_{n}\cap U_0=\bigcup_{0=t_0<t_1<\cdots<t_n}\bigcap_{0\le i\le n}F_{t_i}^{-1}(U_i)$ is nonempty.
			Trivially, $Z_0\cap U_0=U_0$.
			By {semi-transitivity}, there exists $t_{n+1}>t_n$ such that $Z_{n+1}\cap U_0=Z_n\cap U_0\cap F_{t_{n+1}}^{-1}(U_{n+1})$ is nonempty.
			We deduce that $Z_n$ is dense (in $X$), and hence, that $\bigcap_{n\in\N}Z_n$ is comeager.
			This is another expression for the set of points whose orbits visit the sequence $(U_i)$ in the correct order, which is what we had to prove.
			\item Now suppose that for every sequence $(U_i)_{i\ge1}$ of open subsets intersecting $\Omega_\FF$, the set $\bigcup_{t_1<\cdots<t_i<\cdots}\bigcap_{i\ge1}F_{t_i}^{-1}(U_i)$ is comeager.
			Let $U,V$ be nonempty open sets such that $V$ intersects $\Omega_\FF$, and $T\in\N$.
			If one defines $U_{T+1}=V$ and $U_i=X$ if $i\ne T+1$, then our assumption gives that there exists $0<t_1<\cdots<t_i<\cdots$ such that $F_{t_{T+1}}^{-1}(V)=\bigcup_{t_1<\cdots<t_i<\cdots}\bigcap_{i\ge1}F_{t_i}^{-1}(U_i)$ is comeager, which implies that it intersects $U$. Note that $t_{T+1}>T$.
			\item The definition of semi-nonwanderingness can simply be applied to the sequence of open sets constantly equal to $U$.
			\popQED\end{itemize}\end{proof}
	
	The previous proposition applies to semi-wixing systems, but weak semi-mixing has another strong consequence.
	\begin{rmq}\label{r:obliksens}
		Any weakly semi-mixing DS $\FF$ is sensitive or admits a trivial limit set.
	\end{rmq}
	\begin{proof}
		If $\Omega_\FF$ is not trivial, then there are two open subsets $V$ and $V'$ which are at positive distance $\varepsilon>0$ {and intersect $\Omega_\FF$}.
		Then for every $x\in X$ and $\delta>0$, $\B_\delta(x)$ intersects both $F_t^{-1}(V)$ and $F_t^{-1}(V')$, for some $t\in\N$, so that there are points $y$ and $y'$ in it, for which $d(F_t(y),F_t(y'))>\varepsilon$.
		By the triangular inequality, $F_t(x)$ should be at distance at least $\varepsilon/2$ of one of the two, which means that $x$ is not $\varepsilon/2$-stable.
	\end{proof}
	
	%
	The following proposition shows that every CA (and even CA sequence) in an oblique direction is weakly semi-mixing.
	\begin{prop}\label{p:obliktrans}
		If $(F_t)_t$ is a CA sequence and $h$ an oblique curve, then the DS $\FF=(F_t\sigma^{h(t)})_t$ is weakly semi-mixing.
	\end{prop}
	\begin{proof}
		It is enough to prove the property for $U=[u]_m$, $U'=[u']_{m'}$, $V=[v]_n$ and $V'=[v']_{n'}$ four cylinders, such that patterns $v$ and $v'$ appear in $\Omega_{\FF}$, and $m,m',n,n'\in\Z$.
		By extending $v$ and/or $v'$ (into a pattern which still appears in the limit set) and $u$ and/or $u'$, we can suppose that they have the same length pairwise, that $m=m'$ and $n=n'$.
		Suppose, without loss of generality, that $h$ is left-oblique: $-\rpm-\prec h$; in particular, there exists $T\in\N$ such that $\rpm-(T)+h(T)+m>n+\len u$.
		$F_T\sigma^{h(T)}$ is a CA of memory $\rpm-(T)+h(T)$ and anticipation $\rpm+(T)+h(T)$: there exists $w_T\in A^{\len{v}+(\rpm+(T)-\rpm-(T))}$ such that \[F_T^{-1}([v]_m)\supseteq [w_T]_{\rpm-(T)+h(T)+m}~.\]
		Hence, there exist $T\in\N$ and $y\in A^\Z$ such that $y\in [u]_n\cap[w_T]_{\rpm-(T)+h(T)+m}$.
		The same is true for $u'$ and $v'$ (for the same $T$).
	\end{proof}

	We deduce the following.
	\begin{crl}\label{c:obliq}
		Consider DS $\FF=(F_t\sigma^{h(t)})_{t}$ where $(F_t)_{t}$ is a CA sequence and $h$ an oblique curve.
		Then $\tilde\omega_\FF=\Omega_\FF$; it is either sensitive or nilpotent.
	\end{crl}
	\begin{proof}
		The equality is direct from Propositions~\ref{p:obliktrans},~\ref{p:transnwand} and~\ref{p:snwand}.
		Sensitivity comes from Proposition~\ref{p:obliktrans}, Remark~\ref{r:obliksens} and the known fact that the limit set of a CA is trivial if and only if it is nilpotent.
	\end{proof}
	
	\subsection{Almost equicontinuity in two directions}
	The purpose of this subsection is to show that if the CA is almost equicontinuous in two directions of opposite sign, then its generic limit set is finite.
	We essentially reprove \cite[Prop~3.2.3, Prop~3.3.4]{Martinthese} (or the corresponding result for linear directions from \cite{sablik}), but additionally discuss the generic limit set. 
	
	Here is the main lemma for understanding directional dynamics. 
	It is based on the fact that if a word $u$ is blocking along $h'\in\BF$, and $s'$ is the minimal corresponding offset, then in particular for every $t\in\N$ there exists $a_{u,h'}(t)\in A$ such that $\forall z\in[u]_{s'} ,a_{u,h'}(t)=F^t(z)_{h'(t)}$.
	\begin{lem}\label{l:qnilpdir}
		Let $F$ be a CA over $A^\Z$ with blocking words $u$ along $h'\in\BF$ with offset $s'\in\Z$ and $v$ along $h''\in\BF$ with offset $s''\in\Z$, 
		and $q'=h'+\len v-s'$ and $q''=h''-s''$.
		Then for every $z\in[v]_0$ 
		and $j\in\co{q'(t)}{q''(t)}$, $F^t(z)_j=a_{u,h'}(t)$.
	\end{lem}
	In other words, in the orbit of such a configuration $z$ will, after some time, appear such letters that depend only on $u$ (which does not necessarily appear in $z$).
	Of course the symmetric statement is true for $[u]$. 
	Before proving the lemma, here are some remarks on directionnally blocking words.
	Of course the symmetric statements hold for right-blocking words.
	\begin{rmq}\label{r:blocking}~\begin{enumerate}
			\item\label{i:super} From the definition, one can see that if $u$ is blocking for CA $F$ along curve $h$, then any word containing $u$ also.
			\item\label{i:common} If two directions are almost equicontinuous, Proposition~\ref{p:wall} states that they admit blocking words $u$ and $v$.
			From the previous point, they admit a common blocking word $uv$.
			\item\label{i:min}
			From the definition, one can see that: if $u$ is a right-blocking word for CA $F$ along directions $h'$ and $h''$, then also along any direction $h\succeq\min(h',h'')$.
			\item\label{i:max}
			From the previous point and the symmetric statement, if $u$ is right- and left-blocking along directions $h'$ and $h''$, then also along any direction $h\in[\min(h',h''),\max(h',h'')]$.
			\item
			In particular, right- and left-blockingness are preserved by $\sim$ (which is not the case for strong blockingness).
	\end{enumerate}\end{rmq} 
	\begin{proof}[Proof of Lemma~\ref{l:qnilpdir}]
		Since they are left-blocking and right-blocking, respectively, for every $t\in\N$:
		\[ \both{ 
			\forall x,y\in[u]_{s'},x_{\ci{s'}}=y_{\ci{s'}}\implies F^t(x)_{\ci{h'(t)}}=F^t(y)_{\ci{h'(t)}}\\
			\forall x,y\in[v]_{s''},x_{\ic{s''}}=y_{\ic{s''}}\implies F^t(x)_{\ic{h''(t)}}=F^t(y)_{\ic{h''(t)}}~.}\]
		By definition, $\forall z\in[u]_{s'},a_{u,h'}(t)=F^t(z)_{h'(t)}$.
		Now let 
		$j\in\co{q'(t)}{q''(t)}$, so that there is a configuration $y\in[v]_0\cap[u]_{j-h'(t)+s'}$ such that $y_{\ic0}=z_{\ic0}$ and $\sigma^{j-h'(t)}(y)_{\ci{s'}}=z_{\ci{s'}}$. 
		Since $\sigma^{j-h'(t)}(y)$ is in $[u]_{s'}$, we get: $F^t(y)_j=a_{u,h'}(t)$.
		On the other hand, since $v$ is right-blocking along $h''$, 
		$\forall t\in\N, F^t(z)_{\ic{h''(t)-s''}}=F^t(y)_{\ic{h''(t)-s''}}$.
		In particular, we get that $F^t(z)_j=F^t(y)_j=a_{u,h'}(t)$.
		%
	\end{proof}
	
	\begin{prop}\label{p:alldir}
		Let $F$ be a CA with a blocking word $u$ along two distinct directions $h'$ and $h''$, with $h''\not\preceq h'$.
		Along any direction $h\in\BF$, the direct realm $\FN[F,h]{(\dinf{a_{u,h'}(t)})_t}$ includes all $\sigma$-periodic configurations where $u$ appears;
		Moreover, the realm $\AN[F,h]{(\dinf{a_{u,h'}(t)})_t}$ includes:
		\begin{enumerate}
			\item\label{i:uuu} all configurations where $u$ appears infinitely many times on the left and on the right, if $h'\preceq h\preceq h''$; 
			\item\label{i:uu} all configurations where $u$ appears infinitely many times on the right, if $h'\pprec h\preceq h''$; 
			\item all configurations where $u$ appears infinitely many times on the left, if $h'\preceq h\pprec h''$; 
			\item\label{i:u} all configurations where $u$ appears, if $h'\pprec h\pprec h''$. 
		\end{enumerate}
	\end{prop}
	\begin{proof}
		Consider a configuration $z\in[u]_i$, for some $i\in\Z$, with some $\sigma$-period $p\ge1$.
		Let $q'$ and $q''$ be as in Lemma~\ref{l:qnilpdir}. Since $q''\sim h''\not\preceq h'\sim q'$, there exists $T\in\N$ such that $q'(T)\le q''(T)+p$.
		Lemma~\ref{l:qnilpdir} says that then $F^T(z)_j=a_{u,h'}(T)$ for all $j\in\co{i+q'(T)}{i+q''(T)}$, and, by periodicity, for all $j\in\Z$.
		This means that $F^T(z)$ is monochrome, and it is clear that is stays monochrome for $t\ge T$.
		Since, by definition of $a_{u,h'}(t)$, it appears in $F^t(z)$, we deduce that the latter is equal to $\dinf{a_{u,h'}(t)}$.
		\begin{enumerate}
			\item From Point~\ref{i:max} of Remark~\ref{r:blocking}, $u$ is right- and left-blocking along every curve $h\in[\min(h',h''),\max(h',h'')]$.
			So from Proposition~\ref{p:wall}, configurations with infinitely many occurrences of $u$ on the left and on the right are equicontinuous.
			From Proposition~\ref{p:eqan} and the previous point that $\AN[F,h]{(\dinf{a_{u,h'}(t)})_t}\supseteq\FN[F,h]{(\dinf{a_{u,h'}(t)})_t}$ is dense, these equicontinuous configurations must also be in $\AN[F,h]{(\dinf{a_{u,h'}(t)})_t}$.
			\item
			Let $z\in\bigcap_{i\in\Z}\bigcup_{j\ge i}[u]_j$ and $n\in\N$.
			If $h\preceq h''\sim q''$, then there exists $j\ge\max_{t\in\N}h(t)-q''(t)+n$ such that $z\in[u]_j$.
			If $q'\sim h'\pprec h$, then there exists $T\in\N$ such that $\forall t\ge T,q'(t)+j\le h(t)-n$.
			From Lemma~\ref{l:qnilpdir} and these inequalities, for all $t\ge T$, the pattern $F^t(z)_{\cc{h(t)-n}{h(t)+n}}$ is monochrome.
			Since this is true for every $n$, every limit point of the orbit must be monochrome. 
			\item This case is symmetric to the previous one.
			\item Let $z\in[u]_j$ for some $j\in\Z$ and $n\in\N$.
			If $q'\sim h'\pprec h\pprec h''\sim q''$, then there exists $T\in\N$ such that for all $t\ge T$, $h(t)+n\le q''(t)+j$ and $q'(t)+j\le h(t)-n$.
			From Lemma~\ref{l:qnilpdir} and these inequalities, for all $t\ge T$, the pattern $F^t(z)_{\cc{h(t)-n}{h(t)+n}}$ is monochrome.
			\popQED\end{enumerate}\end{proof}

	\subsection{Classification of generic limit sets up to directions}
	We recall the classifications of CA up to shift from \cite{sablik,Martin} and emphasize the properties of each class in terms of generic limit set.
	As the closure of an asymptotic set, the generic limit set may depend on the direction.
	By \dfn{strictly almost equicontinuous}, we mean almost equicontinuous but not equicontinuous.
	\begin{thm}\label{t:class}
		Every CA $F$ satisfies exactly one of the following statements:
		\begin{enumerate}
			\item\label{i:class1} $F$ is \emph{nilpotent}; there is a symbol $a\in A$ such that for all $h\in\BF$, $\tilde\omega_{F,h}=\{\dinf{a}\}$.
			\item\label{i:class2} $F$ is \emph{equicontinuous along a single direction} $h'\in[-\rpm+,-\rpm-]$, and sensitive along other directions; for every $h\in\BF$, $\tilde\omega_{F,h}=\Omega_F$ is infinite.
			\item\label{i:class3} $F$ is strictly \emph{almost equicontinuous along a nondegenerate interval} $S\subseteq[-\rpm+,-\rpm-]$ and sensitive along other directions; there exists $a\in A$ such that $\tilde\omega_{F,h}= 
			\orb_F(\dinf{a})$ for every $h\in S$, and $\tilde\omega_{F,h}$ is infinite for every $h\notin S$;
			moreover, $\E_{F,h}\subseteq\AN[F,h]{\tilde\omega_{F,h}}=\asym[F,h]{\dinf a}$, and if $h\in\oin S$, $\E_{F,h}$ includes a dense open set.
			\item\label{i:class4} $F$ is strictly \emph{almost equicontinuous} \emph{along a single direction} $h'\in[-\rpm+,-\rpm-]$ 
			and sensitive along other directions; for every $h\in\BF$, $\tilde\omega_{F,h}$ is infinite.
			\item[4'.]\label{i:class4p} $F$ is strictly \emph{almost equicontinuous along a single direction} $h'\in[-\rpm+,-\rpm-]$ 
			and sensitive along other directions; $\tilde\omega_{F,h}$ is finite if and only if $h=h'$.
			\item\label{i:class5} $F$ is \emph{sensitive in every direction}; $\tilde\omega_{F,h}$ is infinite along all $h\in\BF$.
	\end{enumerate}\end{thm}
	Compared to \cite[Thm~2.9]{Martin}, we have merged the last two classes, because expansiveness is not relevant in terms of generic limit set, except that it implies surjectivity.
	Surjective CA have their generic limit set equal to the full shift of configurations in every direction: they are either in Class~\ref{i:class2}, Class~\ref{i:class4} or Class~\ref{i:class5}.
	In Class~\ref{i:class3}, each bound of the interval of almost equicontinuity can be included in it or not; actually all four cases can happen: see \cite{Martinthese} for some examples (the bound would be included if one allows directions with unbounded variation).
	\begin{proof}
		\begin{enumerate}
			\item Nilpotent CA have a trivial limit set (in every direction), which includes the generic limit set. 
			\item Now suppose that $F$ is not nilpotent, which is equivalent to $\Omega_F$ being infinite.
			Assume also that there is at least one direction of equicontinuity.
			By Remark~\ref{r:eqlyapu}, all other directions are oblique, hence sensitive by Corollary~\ref{c:obliq}.
			By Proposition~\ref{p:equi} and Corollary~\ref{c:obliq}, $\tilde\omega_{F,h}=\Omega_F$, for all $h\in\BF$.
			\item Now suppose that $F$ admits no direction of equicontinuity, but two distinct directions $h'$ and $h''$ of almost equicontinuity.
			By Point~\ref{i:common} of Remark~\ref{r:blocking}, these two directions have a common blocking word $u$, so that we can apply Proposition~\ref{p:alldir}: $u$ is also blocking for directions in the interval $[\min(h',h''),\max(h',h'')]$.
			Since this is true for every almost equicontinuous $h',h''$, we deduce that the set $S$ of almost equicontinuous directions is convex: it is a nondegenerate interval.
			By Corollary~\ref{c:obliq}, it is included in $[-\rpm+,-\rpm-]$.
			\\
			Proposition~\ref{p:alldir} also states that $\AN[F,h]{(\dinf{a_{u,h'}(t)})_t}$ is dense for every $h\in\BF$, so that it includes all equicontinuous points, by Proposition~\ref{p:eqan}.
			If $h\in S$, Corollary~\ref{c:aleq} gives that the generic limit set is the closure of the asymptotic set of $\E_{F,h}$, that is then the asymptotic set of $\sett{\dinf{a_{u,h'}(t)}}{t\in\N}$, which is a set of monochrome configurations; in particular, it does not depend on $h$.
			By Proposition~\ref{p:periodicorbit}, it is the orbit by $F$ of a monochrome configuration. 
			\\
			For other directions, the generic limit set of a sensitive DS is infinite, by Corollary~\ref{c:sensinfini}.
			\\
			Finally, if $h\in\oin S$, it is easy to find $q',q''\in S$ such that $q'\pprec h\pprec q''$.
			Point~\ref{i:u} of Proposition~\ref{p:alldir} gives that $\AN[F,h]{(\dinf{a_{u,q'}(t)})_t}$ contains a dense open set.
			The same argument as above gives that this is in the realm of the generic limit set. 
			%
			\item The cases remain when there is at most one direction of almost equicontinuity; it cannot be oblique, and other directions have to all have infinite generic limit by Corollary~\ref{c:sensinfini}. This settles the last three classes.
			\popQED\end{enumerate}\end{proof}
	In the following examples, we will meet all 6 classes from the previous classification. 
	\begin{exmp}[Shift]\label{e:sigm}
		Let $\sigma$ be the CA over $A^\Z$ defined by $\sigma(x)_i=x_{i+1}$.
		This CA is reversible, hence surjective. By Corollary \ref{c:surjective}, $\tilde\omega_{\sigma,h}=A^\Z$ along all $h\in\BF$. Along direction ${-1}$, it corresponds to the identity CA. This CA has only one equicontinuous direction: it is in Class~\ref{i:class2} of Theorem~\ref{t:class}.
	\end{exmp}
	
	Let $F$ be a CA with alphabet $A$, memory $r_-\in\Z$, anticipation $r_+>r_-$ and local rule $f$. A state $0\in A$ is \dfn{spreading} if for all $u\in A^{r_+-r_-+1}$ such that $0\sqsubset u$, one has $f(u)=0$.
	\begin{rmq}\label{r:spreading}
		Let $F$ be a 
		CA over $A^\Z$ with memory $r_-\in\Z$, anticipation $r_+>r_-$, and spreading state $0\in A$.
		Then $\tilde\omega_{F,h}=\{\dinf{0}\}$ along all $h\in [{-\rpm+},{-\rpm-}]$ and $\tilde\omega_{F,h}=\Omega_F$ along all $h\notin [{-\rpm+},{-\rpm-}]$, where $\rpm+(t)=r_+t$ and $\rpm-(t)=r_-t$.
		In other words, $F$ is in Class~\ref{i:class3} or Class~\ref{i:class1}.
	\end{rmq}
	Note that any CA is a subsystem of a CA with a spreading state (simply by artificially adding it to the alphabet).
	In particular, unlike the asymptotic set (which includes the nonwandering set), the generic limit set does not support the topological entropy (in the sense of \cite{bowen}).
	\begin{proof}
		We suppose that $F$ has a spreading state $0\in A$. By definition, it is a (left- and right-) blocking word along all $h\in [{-\rpm+},{-\rpm-}]$. By Proposition~\ref{p:alldir}, there exists $a\in A$ such that $\tilde\omega_{F,h}=\orb_F(\dinf a)=\cl{\omega(\E_{F,h})}$ along all $h\in [{-\rpm+},{-\rpm-}]$. In this case, $a$ is nothing else than $0$ (from the definition of $a_{u,h}(t)$). Moreover, since any $h\notin [{-\rpm+},{-\rpm-}]$ is oblique, $\tilde\omega_{F,h}=\Omega_F$ by Corollary~\ref{c:obliq}.
	\end{proof}
	
	\begin{figure}[ht]
		\centering
		\includegraphics[width=6 cm]{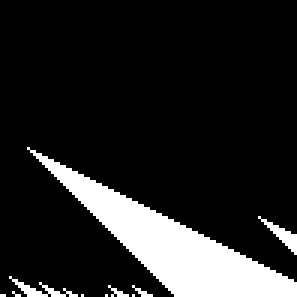}
		\caption{
			{$\Min$ CA along direction $1$};
			$0$ \resp {$1$} are represented by black squares \resp{white squares}.}
		\label{fig:oblique}
	\end{figure}
	
	The simplest example of spreading state is the following.
	\begin{exmp}[$\Min$]\label{e:sigmin}
		Following Example~\ref{e:min}, we define $\Min$ over $\{0,1\}^\Z$ by $\Min(x)_i=\min (x_i,x_{i+1})$.
		A typical space-time diagram is represented in Figure~\ref{fig:oblique}.
		By Remark~\ref{r:spreading}, we know that this CA is in Class~\ref{i:class3}:
		\begin{enumerate}
			\item\label{i:unidmin} Along direction ${0}$:
			\\
			$\tilde\omega_{\Min,h}=\{\dinf{0}\}$ and its realm $\bigcap_{k\in\Z}\sett{x\in \{0,1\}^\Z}{\exists i\geq k,x_i=0}$ is comeager.
			\item Along direction $-1$: the same is true, replacing $\geq$ by $\leq$.
			\item\label{i:Min} Along directions in $]-1,0[$: $\tilde\omega_{\Min,h}=\{\dinf{0}\}$ and its realm is the dense open set $\{0,1\}^\Z\setminus\{\dinf1\}$.
			Note that along direction ${-1}/{2}$: $\Min$ corresponds, up to a power $2$, to the \emph{three-neighbor $\Min$ CA}, defined by $\sigma^{-1}\Min^2(x)_i=\min (x_{i-1},x_i,x_{i+1})$.
			This is an example of uniform DS whose generic limit set does not include all recurrent points.
			\item\label{i:oblmin} Any direction $h\notin [{-1},0]$ is \emph{oblique}: $\tilde\omega_{\Min,h}=\Omega_\Min=\sett{x\in \{0,1\}^{\Z}}{\forall k>0, 10^k1\not\sqsubset x}$.
			The realm of the generic limit set is $\{0,1\}^\Z$.
			This is an example of uniform DS whose set of recurrent points is not dense in the generic limit set.
			\popQED\end{enumerate}\end{exmp}
	
	\begin{exmp}[Lonely gliders]\label{e:lgliders2}
		
		
		
		The CA $F$ from Example~\ref{e:lgliders} is surjective and almost equicontinuous.
		Hence this CA is in Class~\ref{i:class4p}. 
		
	\end{exmp}
	\begin{exmp}[Finite generic limit set]\label{e:onedirection}
		Consider the CA $\Min\times\sigma^{-1}\Min$ defined over $(\{0,1\}^\Z)^2$ by $(\Min\times\sigma^{-1}\Min)(x,y)_i=(\min(x_i,x_{i+1}),\min(y_{i-1},y_i))$.
		According to Example~\ref{e:sigmin}, $\tilde\omega_{\Min,h}=\{\dinf{0}\}$ along all $h\in [{-1},0]$ and $\tilde\omega_{\Min,h}=\Omega_\Min$ along all $h\notin [{-1},0]$; of course $\tilde\omega_{\sigma^{-1}\Min,h}=\{\dinf{0}\}$ along all $h\in [{0},1]$ and $\tilde\omega_{\sigma^{-1}\Min,h}=\Omega_\Min$ along all $h\notin [{0},1]$. Hence, $\tilde\omega_{\Min\times\sigma^{-1}\Min,0}=\{\dinf{0}\}^2$, and $\tilde\omega_{\Min\times\sigma^{-1}\Min,h}= \{\dinf{0}\}\times\Omega_\Min$ \resp{$\tilde\omega_{\Min\times\sigma^{-1}\Min,h}=\Omega_\Min\times\{\dinf{0}\}$} along all $h\in [{-1},0[$ \resp{$h\in]0,1]$}, and $\tilde\omega_{\Min\times\sigma^{-1}\Min,h}=\Omega_\Min^2$ along all $h\notin[{-1},1]$.
		In particular, $\Min\times\sigma^{-1}\Min$ has only one almost equicontinuous direction: $\{0\}$. Hence, it is in Class~\ref{i:class4p}'.
	\end{exmp}
	
	The next two examples are in Class~\ref{i:class5}.	
	\begin{exmp}[Sensitivity in every direction]
		Consider the CA $\sigma^{-1}\Min\times\sigma$ defined over $(\{0,1\}^\Z)^2$ by $(\sigma^{-1}\Min\times\sigma)(x,y)_i=(\min(x_{i-1},x_i),y_{i+1})$. 
		According to Example~\ref{e:sigm}, $\tilde\omega_{\sigma,h}=\{0,1\}^\Z$ along all $h\in\BF$.
		According to Example~\ref{e:sigmin}, $\tilde\omega_{\sigma^{-1}\Min,h}=\{\dinf{0}\}$ along all $h\in[0,1]$ and $\tilde\omega_{\sigma^{-1}\Min,h}=\Omega_\Min$ along all $h\notin[0,1]$.
		Hence, $\tilde\omega_{\sigma^{-1}\Min\times\sigma,h}=\{\dinf{0}\}\times\{0,1\}^\Z$ along all $h\in[0,1]$, and $\tilde\omega_{\sigma^{-1}\Min\times\sigma,h}=\Omega_\Min\times\{0,1\}^\Z$ along all $h\notin[0,1]$.
		Since there is no almost equicontinuous direction, this CA is in Class~\ref{i:class5}.
	\end{exmp}
	
	\begin{figure}[ht]
		\centering
		\includegraphics[width=6 cm]{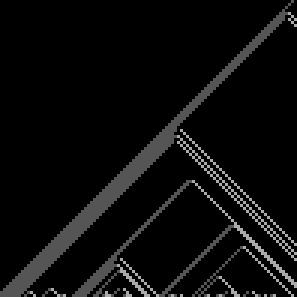}
		\caption{Just Gliders
			($\leftarrow$ are represented by light grey squares and $\rightarrow$ are represented by dark grey squares).}
		\label{fig:glider}
	\end{figure}
	\begin{exmp}[Just Gliders] \label{e:gliders}
		Let $A=\{\leftarrow,0,\rightarrow\}$ and $F$ the CA defined by the following local rule:
		\[f(x_{i-1},x_i,x_{i+1})_i=\soit{
			\rightarrow   &   \si x_{i-1}=\rightarrow \et x_i\ne\leftarrow \et (x_{i+1}\ne\leftarrow \oub x_i=\rightarrow)\\ 
			\leftarrow  &   \si x_{i+1}=\leftarrow \et x_i\ne\rightarrow \et (x_{i-1}\ne\rightarrow \oub x_i=\leftarrow)\\ 
			0   &   \sinon~.}\]
		
		A typical space-time diagram of this CA is shown in Figure~\ref{fig:glider}.
		It is possible to interpret it as a background of $0$s where particles $\rightarrow$ and $\leftarrow$ go to the right and to the left, respectively. When two opposite particles meet they disappear.
		
		One can see 
		that the limit set is $\Omega_F=\sett{x\in A^{\Z}}{\forall k\in\N, \leftarrow0^k\rightarrow\not\sqsubset x}.$
		We prove here that $F$ is weakly semi-mixing in every direction $h\notin\{-1,+1\}$. Hence $\tilde\omega_{F,h}=\Omega_{F,h}$.
		Moreover, $\tilde\omega_{F,-1}=\{0,\rightarrow\}^\Z$, $\tilde\omega_{F,+1}=\{0,\leftarrow\}^\Z$, and $F$ is sensitive in every direction; it is in Class~\ref{i:class5}.
	\end{exmp}
	Similar results have been proved in the measure-theoretical setting in \cite{Maass}.
	The simplicity of the following proof illustrates the power of our setting, and of the notion of semi-mixingness.
	\begin{proof}
		By induction on $t\in\N$, one can see that $F^t(x)_k=\rightarrow$ if and only if $x_{-t+k}=\rightarrow$ and $u=x_{\cc{-t+k+1}{t+k}}$ is a \dfn{right-balanced} pattern, that is, it does not \emph{send} any particle to the left, or more formally:
		\[\forall j\in\co0{\len u}, \sum_{i=0}^j \gamma(u_i)\ge0~, \ou\gamma(u_i)=\soit{
			+1   &   \si u_{i}=\rightarrow \\ 
			0  &   \si u_{i}=0\\
			-1   &   \si u_{i}=\leftarrow}.\]
		Generalizing this induction, we can see that if $k\in\Z$, $t\in\N$, $\leftarrow\not\sqsubset w$ and $u\in A^{2t}$ is right-balanced, then $F^t([wu]_k)\subseteq[w]_{k+t}$.
		We define \dfn{left-balanced} patterns symmetrically, and get that if $\rightarrow\not\sqsubset z$ and $u$ is right-balanced, then $F^t([uz]_k)\subseteq[z]_{k+t}$.
		
		
		
		\begin{itemize}
			\item Let $[u]_m$, $[u']_{m'}$, $[v]_{n}$ and $[v']_{n'}$ be four cylinders, and assume that the last two intersect $\Omega_F$.
			By the expression of $\Omega_F$, note that we can decompose them as $v=wz$, $v'=w'z'$, with $\leftarrow\not\sqsubset w,w'$ and $\rightarrow\not\sqsubset z,z'$. 
			We prove that there is a time step $t\in\N$ such that both $F^t\sigma^{h(t)}([u]_m)\cap[v]_{n}\ne\emptyset$ and $F^t\sigma^{h(t)}([u']_{m'})\cap[v']_{n'}\ne\emptyset$.
			If $h$ is oblique, the result follows from Proposition~\ref{p:obliktrans}, hence we can assume that $h\in]-1,+1[$.
			\\
			We can assume that $u$ and $u'$ are left- and right-balanced pattern: just extend it with the suitable number of $\rightarrow$ on the left or of $\leftarrow$ on the right (the obtained cylinders are included in the original one). { We can also add $0$s to $u$ and $u'$ until being able to assume that they have the same length and that $m=m'$.}
			\\
			Since $h\succ-1$, there exists $t\in\N$ such that $h(t)>-t+\max(n+\len w,n'+\len{w'})-m$. 
			Since $h\prec1$, there exists $t\in\N$ such that $h(t)< t+\min(n+\len w,n'+\len{w'})-m-\len u$. 
			These $t$ could be distinct, but it is not difficult to be convinced that, since $h$ has bounded variation, there is a common $t\in\N$ which satisfies both.
			In that case we can define $\tilde u=0^{t-n+h(t)-\len w+m}u0^{t-m-\len u+n-h(t)+\len w}$.
			Clearly, it is still left- and right-balanced, so that $F^t\sigma^{h(t)}([w\tilde u]_{n-h(t)-t})\subseteq[w]_n$
			and $F^t\sigma^{h(t)}([\tilde uz]_{n-h(t)+\len w-t})\subseteq[z]_{n+\len w}$; taking the intersection we get $F^t\sigma^{h(t)}([w\tilde uz]_{n-h(t)-t})\subseteq[v]_n$.
			Moreover, $[w\tilde uz]_{n-h(t)-t}\subseteq[\tilde u]_{n-h(t)-t+\len w}\subseteq[u]_m$; we get the wanted nonempty intersection.
			The exact same can be achieved for $[u']_m$ and $[v']_{n'}$ for the same $t$.
			%
			%
			\item
			Proposition~\ref{p:snwand} and Remark~\ref{r:obliksens} then give that for every $h\notin\{-1,+1\}$, $\tilde\omega_{F,h}=\Omega_F$ and $F$ is sensitive.
			Since the set of almost equicontinuous directions is an interval, then at least one direction in $\{-1,+1\}$ should also be sensitive.
			Since the definition of the local rule is exactly symmetric, we get that both directions are also sensitive.
			\item
			Now consider direction $h=+1$.
			For $i\in\N$, let $W_i$ be the set of configurations $x\in A^\Z$ such that $x_{\cc1i}$ is not right-balanced.
			If $x\in W_i$, then by definition $x_{\cc1t}$ is not right-balanced, for $t\ge i$, so that, by the first claim of the proof, $F^t\sigma^{t}(x)_0\ne\rightarrow$.
			Since every pattern can be extended to the right into a pattern which is not right-balanced, we see that $W=\bigcup_{i\in\N}W_i$ is a dense open set.
			We get that $\omega_{F,1}(W)\cap[\rightarrow]=\emptyset$.
			Hence $\omega_{F,1}(\bigcap_{n\in\Z}\sigma^n(W))\subseteq\{0,\leftarrow\}^\Z$.
			$\bigcap_{n\in\Z}\sigma^n(W)$ being comeager, we get that $\tilde\omega_{F,1}\subseteq\{0,\leftarrow\}^\Z$.
			Conversely, for every cylinders $[u]_m$ and $[v]_n$, the latter intersecting $\{0,\leftarrow\}^\Z$, the same argument above allows to find arbitrarily large $t\in\N$ such that $F^t([u]_m)$ intersects $[v]_n$, so that 
			{Point~\ref{i:clos} of Remark~\ref{r:geninter}} gives that $[v]_n$ intersects the generic limit set.
			\item
			The exact symmetric argument settles the case of $h=-1$.
			\popQED\end{itemize}\end{proof}

	The following two examples enjoy additional properties that are counter-intuitive, that we state; for a full understanding of these constructions, we refer to the original articles, because each of them could fit a whole article by itself.
	\begin{exmp}\label{e:martin}
		In \cite[Thm~6.1]{Boyer}, a CA $F$ is built with a word $u$ which is blocking in a nondegenerate interval of directions, and $\tilde\omega_F=\omega_F([u]_0)$ is the orbit of a monochrome configuration (see Proposition~\ref{p:alldir}), but here this orbit is nontrivial: in particular, $\tilde\omega_{F^2}=\Omega_{F^2,\mu}\subsetneq\Omega_{F,\mu}=\tilde\omega_F$.
		\\
		This shows that CA which have a finite generic limit set
		are not always \emph{generically nilpotent}: they can converge to a nontrivial orbit.
		This property is possible only if the blocking words are Gardens of Eden.
	\end{exmp}
	
	The following example shows that it is relevant to study arbitrary curves rather than just linear directions.
	\begin{exmp}\label{e:parabole}
		In \cite[Prop~3.3]{Martin}, a CA is built which is almost equicontinuous along $h$ if and only if $0\prec h\prec p$, where $p$ is an explicit function close to the square root function.
		There is a nondegenerate interval of almost equicontinuous directions, but only one of them is linear.
		
		With some horizontal bulking operation and a product with the CA from \cite[Prop~3.1]{Martin}, which deals with the other side of a parabola, one can even obtain a CA which is still almost equicontinuous in two directions of opposite sign,
		but which is sensitive along all linear directions.
	\end{exmp}
	
	\section{Links with the measure-theoretical approach}\label{s:measure}
	By a \dfn{measure}, we mean a Borel probability measure on $X$. The topological \dfn{support} $\supp{\mu}$ of a measure $\mu$ is the smallest closed subset of measure $1$. If $\supp{\mu}=X$, we say that $\mu$ has \dfn{full support}. 
	
	We say that $\FF=(F_t)_{t\in\N}$ is \dfn{$\mu$-equicontinuous} if $\mu(\E_\FF)=1$.
	Proposition~\ref{p:wall} and Corollary~\ref{c:comeager} give that, if $\mu$ is $\sigma$-ergodic, a CA $F$ along some direction $h$ is $\mu$-equicontinuous unless $\E_{F,h}\cap\supp\mu=\emptyset$.
	This generalizes \cite[Prop~3.5]{Gilman}, which was stated only for Bernoulli measures.
	
	\subsection{$\mu$-likely limit set and $\mu$-limit set}
	The generic limit set is the topological variant of the \dfn{$\mu$-likely limit set} $\Lambda_{\FF,\mu}$, which is the smallest closed subset of $X$ that has a realm of attraction of measure one.
	\cite[Ex~5, 6]{milnor} point that there are no general inclusion relations between the two sets, but that they intersect.
	Here is a formalization of this argument.
	\begin{prop}
		For every DS $\FF$ and full-support measure $\mu$, $\Lambda_{\FF,\mu}\cap\tilde\omega_\FF\ne\emptyset$.
	\end{prop}
	\begin{proof}
		Remark that $\Lambda_{\FF,\mu}$ is a closed set with dense (because measure-1) realm. Proposition \ref{p:densinter} allows to conclude.
	\end{proof}
	
	The $\mu$-likely limit set should not be confused with the \dfn{$\mu$-limit set} $\Omega_{\FF,\mu}$, from \cite{Maass,Martin}, which is the intersection of all closed subsets $U$ such that $\lim_{t\to\infty}\mu(F_t^{-1}(U))=1$.
	We prove one general inclusion, though; it is a generalization of \cite[Prop~1]{Maass}.
	\begin{prop}\label{p:mulik}
		For every DS $\FF$ and Borel probability measure $\mu$, $\Omega_{\FF,\mu}\subseteq\Lambda_{\FF,\mu}$.
	\end{prop}
	\begin{proof}
		It is enough to prove that, for every $\varepsilon>0$, $\lim_{t\to\infty}\mu(F_t^{-1}(\cl\B_\varepsilon(\Lambda_{\FF,\mu})))=1$. 
		Suppose, for the sake of contradiction, that it is not the case: there is $\alpha>0$ and an infinite set $I\subseteq\N$ such that for all $t\in I$, $\mu(A_t)\ge\alpha$, where $A_t=F_t^{-1}(\compl{\cl\B_\varepsilon(\Lambda_{\FF,\mu})})$.
		For every $T\in\N$, consider the set $B_T=A_T\setminus\bigcup_{t>T}A_t$ of points which are \emph{for the last time} in some $A_T$.
		Note that the set of those which are \emph{still expecting one visit}, $C_T=\bigcup_{t\in\N}A_t\setminus\bigsqcup_{t\le T}B_t=\bigcup_{t>T}A_t$, includes $A_t$ for some $t\in I$.
		It results that $\mu(C_T)\ge\mu(A_t)\ge\alpha$.
		Since $(C_T)_T$ is a decreasing sequence of subsets, we get that $\mu(C)\ge\alpha$, where $C=\bigcap_{T\in\N}C_T$.
		Now if $x\in C$, then for all $T\in\N$, there exists $t>T$ such that $F_t(x)\in\compl{\cl\B_\varepsilon(\Lambda_{\FF,\mu})}$.
		This means that $x\notin\AN[\FF]{\Lambda_{\FF,\mu}}$, so that $\AN[\FF]{\Lambda_{\FF,\mu}}$ has measure at most $1-\alpha$.
	\end{proof}
	The converse is in general false: Example~\ref{e:gliders} is a counter-example, as proved in \cite{Maass}; it is known to have a $\mu$-limit set which is strictly included in the $\mu$-likely limit set, when $\mu$ is the uniform Bernoulli measure (see \cite[Ex~3]{Maass} and \cite[Ex~4]{KuMa02}). 
	
	However, we prove the converse in a specific case.
	\begin{prop}
		Let $\FF$ be a $\mu$-equicontinuous DS for some measure $\mu$.
		Then $\Omega_{\FF,\mu}=\Lambda_{\FF,\mu}=\cl{\omega_\FF(\E_\FF\cap\supp\mu)}$.
	\end{prop}
	In particular, if $\mu$ has full support, then these sets are equal to $\cl{\omega_\FF(\E_\FF)}=\tilde\omega_\FF$.
	\begin{proof}
		Since $\AN[\FF]{\Lambda_{\FF,\mu}}$ is dense in $\supp\mu$,
		it includes the full-measure set $\E_\FF\cap\supp\mu$, thanks to Proposition~\ref{p:eqan}. In the same way as in the proof of Corollary~\ref{c:aleq}, we can quickly deduce the second equality.
		Now let $y\in\cl{\omega_\FF(\E_\FF\cap\supp\mu
			)}$ and $\varepsilon>0$.
		There is a point $x\in\E_\FF\cap\supp\mu
		$, and a subsequence $(t_n)_n$ such that 
		{$\forall n\in\N,F_{t_n}(x)\in\B_{\varepsilon/2}(y)$.}
		By equicontinuity of $x$, there exists $\delta>0$ such that $\forall t\in\N,F_t(\B_\delta(x))\subseteq\B_{\varepsilon/2}(F_t(x))$.
		In particular, for $n\in\N$, we get $F_{t_n}(\B_\delta(x))\subseteq\B_\varepsilon(y)$, so that $\mu(F_{t_n}^{-1}(\B_\varepsilon(y)))\ge\mu(\B_\delta(x))>0$.
		Let $U$ be a closed subset of $\compl{\B_\varepsilon(y)}$. Then for every $n\in\N$, $\mu(F_{t_n}^{-1}(U))\le\mu(F_{t_n}^{-1}(\compl{\B_\varepsilon(y)}))\le 1-\mu(\B_\delta(x))$.
		Since this is positive and independent of $n$, we get that $\mu(F_{t}^{-1}(U))$ does not converge to $1$.
		By contrapositive, we get that every closed $U$ such that $\mu(F_{t}^{-1}(U))$ converges to $1$ should intersect $\B_\varepsilon(y)$, for every $\varepsilon>0$.
		Hence it contains $y$, and we can conclude that $y\in\Omega_{\FF,\mu}$.
		\\
		The converse inclusion comes from Proposition~\ref{p:mulik}.
	\end{proof}
	
	\subsection{$\mu$-likely limit set of cellular automata}
	If $\FF$ is a DS, we say that a measure $\mu$ is \dfn{$\FF$-invariant} if $\mu(F_t^{-1}(U))=\mu(U)$ for all measurable subsets $U$ of $X$ and all $t\in\N$.
	A $\FF$-invariant measure $\mu$ is \dfn{$\FF$-ergodic} if the measure of every strongly $\FF$-invariant set is either $0$ or $1$.
	This is the measure-theoretical counterpart of the topological notion of transitivity, as emphasized by Lemma~\ref{l:comeager}, which gives the following.
	\begin{crl}\label{c:comeager}
		Let $\FF=(F_t)_{t\in\N}$ be a transitive DS, where all $F_t$ are homeomorphisms, $\mu$ be a full-support ergodic measure, and $W\subseteq X$ a strongly $\FF$-invariant subset.
		Then the following are equivalent.
		\begin{enumerate}
			\item $W$ is not meager.
			\item $W$ is comeager.
			\item $W$ has measure $1$.
			\item $W$ has positive measure.
	\end{enumerate}\end{crl}
	\begin{proof}
		The first two points are equivalent thanks to Lemma~\ref{l:comeager}, and the last two by definition of ergodicity.
		Now if $W$ is comeager, it can be written as a countable intersection $\bigcap_{n\in\N}W_n$ of dense open sets $W_n$, and hence as a countable intersection $\bigcap_{n\in\N}\orb_\FF(W_n)$ of dense open $\FF$-invariant sets. By full support, the measure of each of them should be positive, and by ergodicity, it should be $1$, hence that of $W$ also.
		Otherwise $W$ is meager, so that the previous argument holds for its complement.
	\end{proof}
	
	In \cite{milnor}, J.~Milnor asks for a \emph{good} criterion for equality between the likely and generic limit sets.
	Here is at least a criterion, in the case of cellular automata.
	\begin{crl}
		If $\FF$ is a sequence of CA, and $\mu$ is a full-support $\sigma$-ergodic measure, then $\tilde\omega_\FF=\Lambda_{\FF,\mu}$.
	\end{crl}
	\begin{proof}
		Proposition~\ref{p:subsh} gives that the generic limit set is the intersection of all closed subsets with strongly $\sigma$-invariant comeager realms. The same simple argument shows that the likely limit set is the intersection of all closed subsets with strongly $\sigma$-invariant realms of measure $1$.
		Thanks to Corollary~\ref{c:comeager}, these realms are actually the same, so that the two sets are equal.
	\end{proof}
	It results that all results from the previous sections hold for the $\mu$-likely limit set in that case.
	Actually, even when they are not equal, most results on the generic limit have a parallel result on the likely limit set, which can be proved with the same proof tools.
	

	\section{Conclusion}\label{s:conclusion}
	
	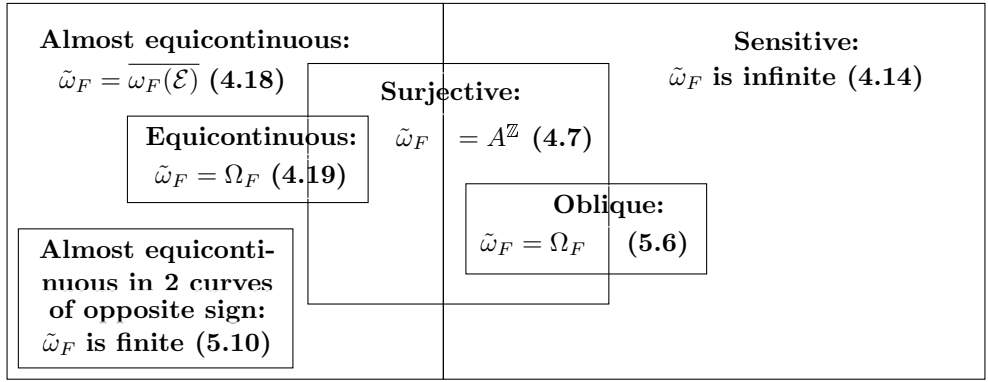
\begin{figure}[ht]\centering
		\begin{tikzpicture}	
		\draw (0,0)--(0,5)--(5.8,5)--(5.8,0)--cycle node[draw=white] at (2.5,4.5) {\bfseries Almost equicontinuous: };
		\node[draw=white] at (2.2,4) {$\tilde{\omega}_F=\cl{\omega_F(\E)}$ \bfseries (\ref{c:aleq})};
		
		\draw (5.8,0)--(13,0)--(13,5)--(5.8,5)--cycle node[draw=white] at (10.5,4.5) {\bfseries Sensitive:};
		
		\node[draw=white] at (10.5,4) {$\tilde{\omega}_F$ \bfseries is infinite (\ref{c:sensinfini})};
		
		\draw (4,1)--(8,1)--(8,4.2)--(4,4.2)--cycle
		node[draw=white] at (5.9,3.8) {\bfseries Surjective:};
		\node[draw=white] at (5.4,3.2) {$\tilde{\omega}_F$};
		
		\node[draw=white] at (6.9,3.2) {$=A^{\Z}$ \bfseries(\ref{c:surjective})};
		
		\draw (0.15,0.15)--(3.8,0.15)--(3.8,2)--(0.15,2)--cycle node[draw=white] at (2,1.7) {\bfseries Almost equiconti-};
		\node[draw=white] at (2,1.3) {\bfseries nuous in 2 curves };
		\node[draw=white] at (2,0.905) {\bfseries of opposite sign:};
		\node[draw=white] at (2,0.465) {$\tilde{\omega}_F$ \bfseries is finite (\ref{t:class})};
		
		\draw (1.6,2.35)--(4.8,2.35)--(4.8,3.5)--(1.6,3.5)--cycle;
		\node[draw=white] at (3.25,3.2) {\bfseries Equicontinuous:};
		\node[draw=white] at (3.25,2.7) {$\tilde{\omega}_F=\Omega_F$ \bfseries(\ref{p:equi})};
		
		\draw (6.1,1.4)--(9.3,1.4)--(9.3,2.6)--(6.1,2.6)--cycle;
		\node[draw=white] at (8,2.3) {\bfseries Oblique: };
		\node[draw=white] at (7,1.8) {$\tilde{\omega}_F=\Omega_F$}; 
		\node[draw=white] at (8.65,1.8) {\bfseries(\ref{c:obliq})};
		\end{tikzpicture}
		\caption  {Summary of the main results.
			Nilpotent CA are not included, for better readability (they would be equicontinuous, almost equicontinuous in 2 directions, oblique, but nonsensitive, and nonsurjective, and have that $\tilde\omega_F$ is a singleton).}
		\label{fig:result}
	\end{figure}
	We studied the generic limit set of dynamical systems, and we emphasized the example of cellular automata. Our main results are:
	\begin{itemize}
		\item The generic limit set of a nonwandering system (in particular, of a surjective CA) is full.
		\item The generic limit set of a semi-nonwandering system (in particular, of an oblique CA) is its limit set.
		\item The generic limit set of an almost equicontinuous system is exactly the closure of the asymptotic set of its set of equicontinuity points.
		\item The generic limit set of an equicontinuous dynamical system is its limit set. 
		\item The generic limit set of a cellular automaton which is almost equicontinuous in two directions of opposite sign is finite; it is the periodic orbit of a monochrome configuration.
		\item The generic limit set of a sensitive system is infinite.
	\end{itemize}
	A summary of these results, for non-nilpotent CA, is represented in Figure~\ref{fig:result}.
	
	Among the interesting questions that the directional classification brings, one can wonder whether, fixing one CA and making the directions vary, we obtain only finitely many generic limit sets, or whether they should intersect, at least as the orbit of a monochrome configuration (which is not clear for the last two classes).
	
	Of course, another natural question is about what happens for two-dimensional CA: in that case almost equicontinuity does not correspond to existence of blocking words, and neither to non-sensitivity, so that everything becomes much more complex.	
	
	\paragraph{Acknowledgements.}
	The authors wish to thank Martin Delacourt and Guillaume Theyssier for fruitful discussions about their results on $\mu$-limit sets and on directional dynamics, and the latter also for his software ACML which helped easily draw our space-time diagrams.
	
	This research has been held while the first author was welcomed in the Institut de Mathématiques de Marseille thanks to a grant {from the Algerian Ministry of High Education and Scientific Research. }
	
	We also thank the reviewer for a careful reading.
	
	\bibliographystyle{plain}
	\bibliography{references}

\end{document}